\documentclass[11pt]{article}

\usepackage{amsmath}
\usepackage{amssymb, yhmath, ulem}
\usepackage[ruled,noend,linesnumbered]{algorithm2e}
\usepackage{float}

\usepackage{enumerate}
\usepackage{listings}

\usepackage[pdftex]{graphicx}
\usepackage{epstopdf}
\usepackage{bm}
\usepackage{authblk}

\usepackage{color}
\definecolor{dkgreen}{rgb}{0,0.6,0}
\definecolor{gray}{rgb}{0.5,0.5,0.5}

\usepackage[bookmarks,pagebackref,pdfpagelabels=true]{hyperref}

\hypersetup{
    colorlinks=true,        % false: boxed links; true: colored
    linkcolor=red,          % color of internal links
    citecolor=green,        % color of links to bibliography
    filecolor=magenta,      % color of file links
    urlcolor=dkgreen           % color of external links
}
\usepackage{kbordermatrix}

\newcommand{\CURID}{{\sf{CUR-ID}}}
\newcommand{\CUR}{{\sf{CUR}}}

\newcommand{\ID}{{\sf{ID}}}
\newcommand{\QR}{{\sf{QR}}}
\newcommand{\SVD}{{\sf{SVD}}}

\newtheorem{theorem}{Theorem}[section]
\newtheorem{corollary}[theorem]{Corollary}
\newtheorem{lemma}[theorem]{Lemma}

\newtheorem{remark}[theorem]{Remark}

\newenvironment{proof}[1][Proof]{\begin{trivlist}\item[\hskip \labelsep {\bfseries #1.}]}{$\Box$\end{trivlist}}

\newcommand{\range}{\operatorname{range}}

\newcommand{\real}{\mathbb R}

\numberwithin{equation}{section}

\newcommand{\mtx}[1]{\bm{\mathsf{#1}}}
\newcommand{\vct}[1]{\bm{\mathsf{#1}}}

\title{Efficient Algorithms for CUR and Interpolative Matrix Decompositions}

\author[1]{Sergey Voronin}
\author[2]{Per-Gunnar Martinsson}
\affil[1]{Department of Mathematics, Tufts University, Medford, MA 02155, USA}
\affil[2]{Department of Applied Mathematics, University of Colorado, Boulder, CO 80309, USA}

\date{\today}

\begin{document}

\maketitle

\abstract{The manuscript describes efficient algorithms for the computation of the \CUR~and \ID~decompositions.
The methods used are based on simple modifications to the classical truncated pivoted
\QR~decomposition, which means that highly optimized library codes can be utilized for implementation.
For certain applications, further acceleration can be attained by incorporating techniques
based on randomized projections.
Numerical experiments demonstrate advantageous performance compared to existing techniques for computing \CUR~factorizations.}

\section{Introduction}

In many applications, it is useful to approximate a matrix $\mtx{A} \in \mathbb{C}^{m \times n}$
by a factorization of rank $k < \min(m,n)$.
When the singular values of $\mtx{A}$ decay sufficiently fast so that an accurate 
approximation can be obtained for a rank $k$ that is substantially smaller than 
either $m$ or $n$,
great savings can be obtained both in terms of storage requirements, and in terms
of speed of any computations involving $\mtx{A}$.
A low rank approximation that is in many ways optimal is the truncated
\textit{singular value decomposition (SVD)} of
rank $k$, which approximates $\mtx{A}$ via the product
\begin{equation}
\label{eq:SVD}
\begin{array}{ccccc}
\mtx{A} &\approx & \mtx{U}_k & \mtx{\Sigma}_k & \mtx{V}^*_k,\\
m\times n && m\times k & k\times k & k\times n
\end{array}
\end{equation}
where the columns of the orthonormal matrices $\mtx{U}_k$ and $\mtx{V}_k$ are the left and
right singular vectors of $\mtx{A}$, and where $\mtx{\Sigma}_k$ is a diagonal matrix holding
the singular values of $\mtx{A}$.
However, a disadvantage of the low rank SVD is its storage requirements. Even
if $\mtx{A}$ is a sparse matrix, $\mtx{U}_k$ and $\mtx{V}_k$ are usually dense.
This means that if $\mtx{A}$ is large and very sparse, compression via the SVD is only
efficient when the rank $k$ is \textit{much} smaller than $\min(m,n)$.

As an alternative to the SVD, the so called \textit{CUR-factorization} \cite{goreinov1997theory,tyrtyshnikov2000incomplete,MR2475795} has recently received much 
attention \cite{Mitrovic_curdecomposition, MR3121656}.
The \CUR-factorization approximates an $m\times n$ matrix $\mtx{A}$ as a product
\begin{equation}
\label{eq:CUR}
\begin{array}{ccccc}
\mtx{A} &\approx & \mtx{C} & \mtx{U} & \mtx{R},\\
m\times n && m\times k & k\times k & k\times n
\end{array}
\end{equation}
where $\mtx{C}$ contains a subset of the columns of $\mtx{A}$ and $\mtx{R}$ contains
a subset of the rows of $\mtx{A}$. The key advantage of the CUR is that
the factors $\mtx{C}$ and $\mtx{R}$ (which are typically much larger than $\mtx{U}$)
inherit properties such as \textit{sparsity} or \textit{non-negativity}
from $\mtx{A}$. Also, the index sets that point out which columns and rows of $\mtx{A}$
to include in $\mtx{C}$ and $\mtx{R}$ often assist in \textit{data interpretation.}
Numerous algorithms for computing the \CUR~factorization have been proposed
(see e.g. \cite{MR2443975,MR3121656}), with some of the most
recent and popular approaches relying on a method known as
leverage scores \cite{MR2443975,MR2475795}, a notion originating
from statistics \cite{HoaglinWelsch}.

A third factorization which is closely related to the \CUR~is
the so called \textit{interpolative decomposition (ID)}, which
decomposes $\mtx{A}$ as
\begin{equation}
\label{eq:CUR2}
\begin{array}{ccccc}
\mtx{A} &\approx & \mtx{C} & \mtx{V}^*,\\
m\times n && m\times k & k\times n
\end{array}
\end{equation}
where again $\mtx{C}$ consists of $k$ columns of $\mtx{A}$. The matrix $\mtx{V}$ contains a
$k\times k$ identity matrix as a submatrix and can be constructed so that
$\max_{i,j}|\mtx{V}(i,j)| \leq 1$, making $\mtx{V}$ fairly well-conditioned.
Of course, one could equally well express $\mtx{A}$ as
\begin{equation}
\label{eq:CUR2rows}
\begin{array}{ccccc}
\mtx{A} &\approx & \mtx{W} & \mtx{R},\\
m\times n && m\times k & k\times n
\end{array}
\end{equation}
where $\mtx{R}$ holds $k$ rows of $\mtx{A}$, and the properties of $\mtx{W}$ are analogous to those
of $\mtx{V}$.
A third variation of this idea is the
\textit{two-sided interpolative decomposition (tsID)},
which decomposes $\mtx{A}$ as the product
\begin{equation}
\label{eq:CUR2b}
\begin{array}{ccccc}
\mtx{A} &\approx & \mtx{W} & \mtx{A}_{\rm skel} & \mtx{V}^*,\\
m\times n && m\times k & k\times k & k\times n
\end{array}
\end{equation}
where $\mtx{A}_{\rm skel}$ consists of a $k\times k$ submatrix of $\mtx{A}$.
The two sided \ID~allows for data interpretation in a manner entirely analogous to the \CUR,
but has an advantage over the CUR in that it is inherently better conditioned,
cf.~Remark \ref{remark:condCUR}.
On the other hand, the factors $\mtx{W}$ and $\mtx{V}$ do not inherit properties such as sparsity or non-negativity.
This makes the two-sided ID only marginally better than the SVD in terms of
storage requirements for sparse matrices.

In this manuscript, we describe a set of efficient algorithms for computing approximate
\ID~and \CUR~factorizations. The algorithms are obtained via slight variations on
the classical ``rank-revealing QR'' factorizations \cite{MR882441} and are easy to
implement---the most expensive parts of the computation can be executed using highly
optimized standard libraries such as, e.g., LAPACK \cite{LAPACK}.
We also demonstrate how the computations can be accelerated by using
randomized algorithms \cite{MR2806637}. For instance, randomization allows us
to improve the asymptotic complexity of computing the CUR decomposition from $O(mnk)$
to $O(mn\,log(k) + (m+n)k^{2})$. Section \ref{sec:num}
illustrates via several numerical examples that the techniques described here for
computing the \CUR~factorization compare favorably in terms of both speed and accuracy
with recently proposed \CUR~implementations. All the \ID~and \CUR~factorization algorithms 
discussed in this article are efficiently 
implemented as part of the open source RSVDPACK package \cite{voronin2015rsvdpack}. 

%The manuscript also investigates the effect of accelerating existing schemes
%for constructing the \CUR-factorization by using the randomized algorithm of
%\cite{MR2806637} to compute the singular vectors used to compute leverage scores.
%
%%\begin{remark}
%\textit{Remark}.
%The \CUR~factorization is closely related to the so called \textit{interpolative
%decomposition (ID)}, which decomposes $\mtx{A}$ as a product
%\begin{equation}
%\label{eq:CUR2}
%\begin{array}{ccccc}
%A &\approx & X & A_{\rm skel} & Y^*,\\
%m\times n && m\times k & k\times k & k\times n
%\end{array}
%\end{equation}
%where $A_{\rm skel}$ consists of a $k\times k$ submatrix of $\mtx{A}$.
%The \ID~allows for data interpretation in a manner entirely analogous to the \CUR,
%but has a great advantage over the CUR in that it is inherently better conditioned.
%However, while the factors $X$ and $Y$ (which tend to require much more storage
%than $A_{\rm skel}$) are well-conditioned, they do not inherit properties such
%as sparsity or non-negativity.
%%Further sections provides a longer discussion of the
%%relative merits of the two factorizations.
%

\section{Preliminaries}
\label{sec:prel}
In this section we review some existing matrix decompositions,
notably the pivoted \QR~, \ID~, and \CUR~ decompositions \cite{MR2806637}.
%The decomposition techniques will be the building blocks for
%the \CURID~decomposition that we introduce in Section \ref{sect:CUR_ID}.
We follow the notation of \cite{MR3024913} (the so called ``Matlab style notation''):
given any matrix $\mtx{A}$ and (ordered) subindex sets $I$ and $J$,
$\mtx{A}(I,J)$ denotes the submatrix
of $\mtx{A}$ obtained by extracting the rows and columns of $\mtx{A}$
indexed by $I$ and $J$, respectively;
and $\mtx{A}(:, J)$ denotes the submatrix of $\mtx{A}$
obtained by extracting the columns of $\mtx{A}$ indexed by $J$.
For any positive integer $k$, $1:k$ denotes the ordered index set $(1,\ldots,k)$.
We take $\| \cdot \|$ to be the spectral or operator norm (largest singular value) and
$\| \cdot \|_{F}$ the Frobenius norm:
$\|x\|_{F} = \left( \displaystyle\sum_{k=1}^n |x_k|^2 \right)^{\frac{1}{2}}$.
\subsection{The singular value decomposition (SVD)}
\label{sec:SVD}
The SVD was introduced briefly in the introduction. Here we define it again, with some
more detail added. Let $\mtx{A}$ denote an $m\times n$ matrix, and set $r = \min(m,n)$.
Then $\mtx{A}$ admits a factorization
\begin{equation}
\label{eq:svd}
\begin{array}{ccccccccccccccccccccc}
\mtx{A} &=& \mtx{U} & \mtx{\Sigma} & \mtx{V}^{*},\\
m\times n && m\times r & r \times r & r\times n
\end{array}
\end{equation}
where the matrices $\mtx{U}$ and $\mtx{V}$ are orthonormal, and $\mtx{\Sigma}$ is diagonal.
We let $\{\vct{u}_{i}\}_{i=1}^{r}$ and $\{\vct{v}_{i}\}_{i=1}^{r}$ denote the columns of
$\mtx{U}$ and $\mtx{V}$, respectively. These vectors are the left and right singular vectors
of $\mtx{A}$. As in the introduction, the diagonal elements $\{\sigma_{j}\}_{j=1}^{r}$ of
$\mtx{\Sigma}$ are the singular values of $\mtx{A}$. We order these so that
$\sigma_{1}  \geq \sigma_{2} \geq \cdots \geq \sigma_{r} \geq 0$.
We let $\mtx{A}_{k}$ denote the truncation of the SVD to its first $k$ terms,
$\mtx{A}_{k} = \sum_{i=1}^{k}\sigma_{i}\,\vct{u}_{i}\,\vct{v}_{j}^{*}$.
It is easily verified that
\begin{equation}
\label{eq:minerrors}
\|\mtx{A} - \mtx{A}_{k}\| = \sigma_{k+1},
\qquad\mbox{and that}\qquad
\|\mtx{A} - \mtx{A}_{k}\|_{F} = \left(\sum_{j=k+1}^{\min(m,n)} \sigma_{j}^{2}\right)^{1/2}.
\end{equation}
Moreover, the Eckart-Young theorem \cite{1936_eckart_young} states that these errors are
the smallest possible errors that can be incurred when approximating $\mtx{A}$ by
a matrix of rank $k$.

\subsection{Pivoted \QR~factorizations}
Let $\mtx{A}$ be an $m\times n$ matrix with real or complex entries,
and set $r = \min(m,n)$. The (compact) \QR-factorization of $\mtx{A}$
then takes the form
\begin{equation}
\label{eq:APQR}
\begin{array}{ccccccc}
\mtx{A} & \mtx{P} &=& \mtx{Q} & \mtx{S},\\
m\times n & n\times n && m\times r & r\times n
\end{array}
\end{equation}
where $\mtx{P}$ is a permutation matrix, $\mtx{Q}$ has orthonormal columns,
and $\mtx{S}$ is upper triangular (the matrix we call ``$\mtx{S}$'' is customarily
labeled ``$\mtx{R}$'', but we use that letter for one of the factors in the
\CUR-decomposition). The permutation matrix $\mtx{P}$ can more efficiently
be represented via a vector $J \in \mathbb{Z}_{+}^{n}$ of indices
such that $\mtx{P} = \mtx{I}(:,J)$ where $\mtx{I}$ is the $n\times n$ identity matrix.
The factorization (\ref{eq:APQR}) can then be written
\begin{equation}
\label{eq:QR}
\begin{array}{ccccccc}
\mtx{A}(:,J) &=& \mtx{Q} & \mtx{S}.\\
m\times n && m\times r & r\times n
\end{array}
\end{equation}
The \QR-factorization is often computed via column pivoting combined
with either the Gram-Schmidt process, Householder reflectors \cite{MR3024913},
or Givens rotations \cite{MR882441}. The resulting factor $\mtx{S}$ then satisfies various
decay conditions \cite{MR3024913},
such as:
$$
\mtx{S}(j,j)\geq \|\mtx{S}(j:m, \ell)\|_2\qquad\mbox{for all}\ j<\ell.
$$

The \QR-factorization (\ref{eq:QR}) expresses $\mtx{A}$
as a sum of $r$ rank-one matrices
$$
\mtx{A}(:,J) \approx \sum_{j = 1}^{r} \mtx{Q}(:,j)\,\mtx{S}(j,:).
$$
The \QR-factorization is often built incrementally via a greedy
algorithm such as column pivoted Gram-Schmidt. This opens up the
possibility of stopping after the first $k$ terms have been computed
and settling for a ``partial \QR-factorization of $\mtx{A}$''. We can
express the error term by splitting the factors in (\ref{eq:QR})
as follows:
\begin{equation}
\label{eqn:trunc}
%\kbordermatrix{&
%n\\
%m & A} \times
%\kbordermatrix{&
%n\\
%n & P}
\mtx{A}(:, J)
\quad = \quad
\kbordermatrix{&
k & r-k\\
m &\mtx{Q}_1&\mtx{Q}_2}
\times
%\kbordermatrix{&
%k & n-k\\
%k &\mtx{S}_{11} & \mtx{S}_{12}\\
%r-k &0 &S_{22}}.
\kbordermatrix{
 & n
\\ k & \mtx{S}_1
\\ r-k & \mtx{S}_2
}
\ =\
\mtx{Q}_1 \mtx{S}_1 + \mtx{Q}_2 \mtx{S}_2.
%\ \approx\
% \mtx{Q}_1 \mtx{S}_1.
\end{equation}
Observe that since the \SVD~is optimal, it is always the case that
$$
\sigma_{k+1}(\mtx{A}) \leq \|\mtx{Q}_2\,\mtx{S}_2\| = \|\mtx{S}_2\|.
$$
We say that a factorization is a ``rank-revealing  \QR-factorization (RRQR)'' if the
ratio $\frac{\|\mtx{S}_2\|}{\sigma_{k+1}(\mtx{A})}$ is guaranteed to be bounded
\cite{GuEisenstat1}.
(Some authors require additionally that $\sigma_{j}(\mtx{S}_1) \approx
\sigma_{j}(\mtx{A})$ for $1 \leq j \leq k$). Classical column pivoted Gram-Schmidt
\textit{typically} results in an RRQR, but there are counter-examples. More
sophisticated versions such as \cite{GuEisenstat1} provably compute an RRQR, 
but are substantially harder to code, and the gain compared
to standard methods is typically modest.

%The rank $k$ matrix $\mtx{Q}_1 \mtx{S}_1$ can serve as a rank $k$ approximation
%Using the fact that $S$ is upper triangular with a
%nonincreasing diagonal (resulting from the column pivoting)
%If we define
%\begin{equation*}
%\mtx{S}_1 =\ \kbordermatrix{&
%k & n-k\\
%k &\mtx{S}_{11} & \mtx{S}_{12}}
%\quad \mbox{and} \quad
%S_2 =\ \kbordermatrix{&
%k & n-k\\
%k &0 & S_{22}} ,
%\end{equation*}
%then, $\mtx{A}(:, J) = \mtx{Q}_1 \mtx{S}_1 + \mtx{Q}_2 S_2 \approx \mtx{Q}_1 \mtx{S}_1$.

\subsection{Low rank interpolative decomposition}
\label{sect:id}

An approximate rank $k$ interpolative decomposition (\ID) of a
matrix $\mtx{A}\in\mathbb{C}^{m\times n}$
is the approximate factorization:
\begin{equation}
\label{eq:defID1}
\begin{array}{cccc}
\mtx{A} &\approx& \mtx{C} &\mtx{V}^*,\\
m\times n && m\times k & k\times n
\end{array}
\end{equation}
where the partial column skeleton
$\mtx{C}\in\mathbb{C}^{m\times k}$
is given by a subset of the columns of $\mtx{A}$
and $\mtx{V}$ is well-conditioned in a sense that we will make precise shortly.
The interpolative decomposition
approximates $\mtx{A}$ using only some of its columns,
and one of the advantages of doing so is that the
more compact description of the range of $\mtx{A}$ given
by its skeleton preserves some of the properties
of the original matrix $\mtx{A}$ such as sparsity and non-negativity.
In this section we show one way of obtaining
a low rank interpolative decomposition,
via the truncated \QR~with column pivoting.

From \eqref{eqn:trunc}, we see that as long as
$\|\mtx{S}_2\|_2$ is small, we can approximate
$\mtx{A}(:, J)$ by $\mtx{Q}_1 \mtx{S}_1$.
We show that the approximation term $\mtx{Q}_1 \mtx{S}_1$
provides a rank $k$ \ID~to the matrix $\mtx{A}$.
In fact, the approximation term $\mtx{Q}_1 \mtx{S}_1$
is the image of a skeleton of $\mtx{A}$, i.e.,
the range of $\mtx{Q}_1 \mtx{S}_1$ is contained in the span
of $k$ columns of $\mtx{A}$.
Splitting the columns of $\mtx{S}_1$ and $\mtx{S}_2$ as follows:
\begin{equation}
\label{eqn:Ssplit}
\mtx{S}_1 =\ \kbordermatrix{&
k & n-k\\
k &\mtx{S}_{11} & \mtx{S}_{12}}
\quad \mbox{and} \quad
\mtx{S}_2 =\ \kbordermatrix{&
k & n-k\\
r-k &\mtx{0} & \mtx{S}_{22}} ,
\quad(\text{i.e., }\ \
\mtx{S} =
\kbordermatrix{&
k & n-k\\
k &\mtx{S}_{11} & \mtx{S}_{12}\\
r-k &\mtx{0} &\mtx{S}_{22}},\ )
\end{equation}
it is immediate that
\begin{equation*}
\mtx{A}(:, J)
\ = \
\mtx{Q}_1 \begin{bmatrix} \mtx{S}_{11} & \mtx{S}_{12} \end{bmatrix}+
\mtx{Q}_2 \begin{bmatrix} 0 & \mtx{S}_{22} \end{bmatrix}
\ = \
\kbordermatrix{&
k & n-k\\
m &\mtx{Q}_1 \mtx{S}_{11} & \mtx{Q}_1 \mtx{S}_{12} + \mtx{Q}_2 \mtx{S}_{22}
}.
\end{equation*}
In other words, we see that the matrix $\mtx{Q}_1 \mtx{S}_{11}$ equals the
first $k$ columns of $\mtx{A}(:, J)$.
We now define the factor $\mtx{C}$ in (\ref{eq:defID1}) via
\begin{equation*}
   \mtx{C}
\ :=\
\mtx{A}(:,J(1:k))
\ = \
\mtx{Q}_1 \mtx{S}_{11}.
\end{equation*}
Then the dominant term $\mtx{Q}_1 \mtx{S}_1$ in (\ref{eqn:trunc}) can be written
\begin{equation*}
\mtx{Q}_1 \mtx{S}_1
\ =\
\begin{bmatrix}
\mtx{Q}_1 \mtx{S}_{11} & \mtx{Q}_1 \mtx{S}_{12} \end{bmatrix}
\ = \
\mtx{Q}_1 \mtx{S}_{11}  \, [\mtx{I}_k \quad \mtx{T}_l].
\ = \
\mtx{C}  \, [\mtx{I}_k \quad \mtx{T}_l],
\end{equation*}
where $\mtx{T}_l$ is a solution to the matrix equation
\begin{equation}
\label{eq:S11T=S12}
\mtx{S}_{11} \mtx{T}_l = \mtx{S}_{12}.
\end{equation}
The equation (\ref{eq:S11T=S12}) obviously has a solution whenever $\mtx{S}_{11}$ is non-singular.
If $\mtx{S}_{11}$ is singular, then one can show that $\mtx{A}$ must necessarily
have rank $k'$ less than $k$, and the bottom $k-k'$ rows in (\ref{eq:S11T=S12}) consist
of all zeros, so there exists a solution in this case as well.
We now recover the factorization (\ref{eq:defID1}) upon setting
\begin{equation}
\label{eqn:column_approx_id2}
{\mtx{V}}^*=\begin{bmatrix} \mtx{I}_k & \mtx{T}_l\end{bmatrix} \mtx{P}^*.
\end{equation}
The approximation error of the \ID~obtained via truncated
\QR~with pivoting is the same as that of the truncated \QR:
\begin{equation}
\label{eq:id_error}
\mtx{A} - \mtx{C} \mtx{V}^* = \mtx{Q}_2 \mtx{S}_{22}
\end{equation}

\begin{remark}
This section describes a technique for converting a QR decomposition of $\mtx{A}$
into the interpolative decomposition (\ref{eq:CUR2}). By applying an
analogous procedure to the adjoint $\mtx{A}^{*}$ of $\mtx{A}$, we obtain the sibling
factorization (\ref{eq:CUR2rows}) that uses a sub-selection of rows of $\mtx{A}$ to
span the row space. In other words, to find the column skeleton, we perform
Gram-Schmidt on the columns on $\mtx{A}$, and in order to find the row skeleton,
we perform Gram-Schmidt to the rows of $\mtx{A}$.
%In addition to the factorization
%Notice also that the approximation $A \approx C V^{*}$ where
%$\mtx{C}$ contains a subset of the columns of $\mtx{A}$, can be replaced by a similar
%approximation using a subset of the rows of $\mtx{A}$. Such a factorization can be
%obtained by performing pivoted \QR~on $A^{*}$ instead of $\mtx{A}$. That is,
%we can obtain:
%\begin{equation*}
%A^{*} \approx A^{*}(:,I(1:k)) Z^{*}
%\end{equation*}
%from an \ID~of $A^{*}$. This implies:
%\begin{equation*}
%A \approx Z \mtx{A}(I(1:k),:).
%\end{equation*}
%In this approximation, we are representing $\mtx{A}$ via a partial row skeleton.
\end{remark}

\subsection{Two sided interpolative decomposition}
\label{sec:twoID}

A two sided \ID~approximation for matrices,
is constructed via two successive one sided \ID s.
Assume that we have performed the one sided decomposition to obtain
\eqref{eqn:column_approx_id2}.
Then perform an ID of the adjoint of $\mtx{C}$ to determine a matrix $\mtx{W}$ and
an index vector $I$ such that
\begin{equation}
\label{eq:boat}
\begin{array}{cccccccccccccccc}
\mtx{C}^{*} &=& \mtx{C}(I(1:k),:)^{*} & \mtx{W}^{*}.\\
k \times m && k\times k & k\times m
\end{array}
\end{equation}
In other words, the index vector $I$ is obtained by performing a
pivoted Gram-Schmidt process on the rows of $\mtx{C}$.
Observe that the factorization (\ref{eq:boat}) is exact since it is a \textit{full}
(as opposed to \textit{partial}) QR factorization.
We next insert (\ref{eq:boat}) into (\ref{eq:defID1}), using that
$\mtx{C}(I(1:k),:) = \mtx{A}(I(1:k),J(1:k))$, and obtain
%Then we can perform a rank $k$ \ID~on $C^{*}$:
%that is, there exist a row permutation $I$ and some $\W\in\mathbb{C}^{m \times k}$
%such that $C^{*} = C^{*}(:,I(1:k)) \W^{*}$
%(where equality holds because $\rank(C)\leq k$). This essentially amounts to
%doing a full pivoted \QR~decomposition on $C^{*}$ which can often by efficiently done
%via canned routines in numerical libraries. Hence
%\begin{equation*}
%%C^{*} = A^{*}_{\sk}(:,I(1:k)) \W^{*}
%%\implies
%C
%\ =\
%\W C(I(1:k),:)
%\ =\
%\W \mtx{A}(I(1:k),J(1:k)).
%\end{equation*}
%%Notice that we have equalities for all the terms above because,
%%$\mtx{C}$ is of size $m \times k$. Hence taking the rank k \ID, amounts
%%to doing the rank $k$ pivoted \QR~on $\mtx{C}$, which exactly represents the matrix.
%Thus, we obtain the approximation to the original matrix $\mtx{A}$:
\begin{equation}
\label{eqn:two_sided_approx_id}
\mtx{A}
\ \approx\ \mtx{C} \mtx{V}^{*} =
\mtx{W} \mtx{A}(I(1:k),J(1:k)) \mtx{V}^{*}.
\end{equation}
%Note that the approximation error in \eqref{eqn:two_sided_approx_id} is
%the same as for the one sided \ID~in \eqref{eqn:column_approx_id2}:
%the product term $\mtx{Q}_2 \mtx{S}_2$ from the truncated \QR~decomposition we performed.

We observe that the conversion of the single-sided ID \eqref{eqn:column_approx_id2}
into the two-sided ID \eqref{eqn:two_sided_approx_id} is \textit{exact} in the
sense that no additional approximation error is incurred:
$$
\mtx{A} - \mtx{C}\,\mtx{V}^{*} = \mtx{A} - \mtx{W}\,\mtx{A}(I(1:k),J(1:k))\,\mtx{V}^{*} = \mtx{Q}_{2} \mtx{S}_{2}.
$$

\begin{remark}
The index vector $I$ and the basis matrix $\mtx{W}$ computed
using the approach described in this section form an approximate
row-ID for $\mtx{A}$ in the sense that $\mtx{A} \approx \mtx{W}\,\mtx{A}(I,:)$.
However, the resulting error tends to be slightly higher than the
error incurred if Gram-Schmidt is performed directly on the rows
of $\mtx{A}$ (rather than on the rows of $\mtx{C}$), cf.~Lemma \ref{lem:Etilde}.
\end{remark}

%Notice also, that the two sided \ID~approach provides another way to
%obtain a row skeleton decomposition. Since $A \approx C V^{*}$, if we take a
%subset of the rows of $\mtx{A}$, via the index set $J_k$ obtained from
%doing an \ID~on $C^{*}$, we get:
%\begin{equation*}
%C(I(1:k),:) V^{*} \approx \mtx{A}(I(1:k),:).
%\end{equation*}
%If we set $R:=\mtx{A}(I(1:k),:)$ and left-multiply both sides by $\mtx{W}$, we obtain:
%\begin{equation*}
%W C(I(1:k),:) V^{*} \approx W \mtx{A}(I(1:k),:) = W R,
%\end{equation*}
%but the left hand side $W C(I(1:k),:) V^{*} = C V^{*} \approx A$,
%so that:
%\begin{equation}
%A \approx W R
%\label{eq:wr_decomp1}
%\end{equation}
%with $\mtx{R}$ a partial row skeleton of $\mtx{A}$.

\subsection{The \CUR~Decomposition}
A rank $k$ \CUR~factorization of a matrix $\mtx{A}\in\mathbb{C}^{m\times n}$ is given by
$$
\begin{array}{ccccccc}
   \mtx{A} &\approx& \mtx{C} &\mtx{U} &\mtx{R},\\
m\times n && m\times k & k\times k & k\times n
\end{array}
$$
where $\mtx{C}$ consists of $k$ columns of $\mtx{A}$, and $\mtx{R}$ consists of
$k$ rows of $\mtx{A}$.
The decomposition is typically obtained
in three steps \cite{Mitrovic_curdecomposition}. First, some scheme is used to assign a
weight or the so called leverage score (of importance) to each column and row in the matrix.
This is typically done either using the $\ell_2$ norms of the columns and rows or
by using the leading singular vectors of $\mtx{A}$ \cite{MR2443975}.
Next, the matrices $\mtx{C}$ and $\mtx{R}$ are constructed via a randomized sampling procedure,
using the leverage scores to assign a sampling probability to each column and row.
Finally, the $\mtx{U}$ matrix is computed via:
\begin{equation}
\label{eq:CURisbad}
\mtx{U} \approx \mtx{C}^{\dagger} \mtx{A} \mtx{R}^{\dagger},
\end{equation}
with $\mtx{C}^{\dagger}$ and $\mtx{R}^{\dagger}$ being the pseudoinverses
of $\mtx{C}$ and $\mtx{R}$.

%We note that as a consequence of this computation, the matrix $\mtx{U}$ is often significantly
%more badly conditioned than $\mtx{A}$.

Many techniques for computing \CUR~factorizations have been proposed. In particular, we mention the recent work
of Sorensen and Embree \cite{2014arXiv1407.5516S} on the DEIM-CUR method. A number of
standard \CUR~algorithms is implemented in the software package rCUR \cite{rCUR} which we use
for our numerical comparisons. The methods in the rCUR package utilize eigenvectors to assign
weights to columns and rows of $\mtx{A}$. Computing the eigenvectors exactly amounts to doing
the \SVD~which is very expensive. However, instead of the full \SVD, when a \CUR~of rank $k$
is required, we can utilize instead the randomized \SVD~algorithm \cite{MR2806637} to compute
an approximate \SVD~of rank $k$ at substantially lower cost.

\begin{remark}[Conditioning of CUR]
\label{remark:condCUR}
For matrices whose singular value experience substantial decay, the
accuracy of the CUR factorization can deteriorate due to effects of
ill-conditioning. To simplify slightly, one would normally expect the
leading $k$ singular values of $\mtx{C}$ and $\mtx{R}$ to be of roughly the same order
of magnitude as the leading $k$ singular values of $\mtx{A}$.
Since low-rank factorizations are most useful when applied to matrices whose
singular values decay reasonably rapidly, we would \textit{typically} expect
$\mtx{C}$ and $\mtx{R}$ to be highly ill-conditioned, with condition numbers roughly
on the order of $\sigma_{1}(\mtx{A})/\sigma_{k}(\mtx{A})$. Hence, in the typical case,
evaluation of the formula (\ref{eq:CURisbad}) can be expected to result in
substantial loss of accuracy due to accumulation of round-off errors.
Observe that the ID does not suffer from this problem; in (\ref{eq:CUR2b}),
the matrix $\mtx{A}_{\rm skel}$ tends to be ill-conditioned, but it does not
need to be inverted. (The matrices $\mtx{W}$ and $\mtx{V}$ are well-conditioned.)
\end{remark}

\section{The \CURID~algorithm}
\label{sect:CUR_ID}

In this section, we demonstrate that the \CUR~decomposition can easily be constructed from
the basic two-sided \ID~(which in turn, recall, can be built from a column
pivoted QR factorization), via a procedure we call ``\CURID''.
The difference between recently popularized algorithms for \CUR~computation and \CURID~is in
the choice of columns and rows of $\mtx{A}$ for forming $\mtx{C}$ and $\mtx{R}$.
In the \CURID~algorithm, the columns and rows are chosen via the two sided \ID.
The idea behind the use of \ID~for obtaining the \CUR~factorization
is that the matrix $\mtx{C}$ in the \CUR~factorization is
immediately available from the \ID~
(see \eqref{eqn:column_approx_id2}),
and the matrix $\mtx{V}\in\mathbb{C}^{n\times k}$
not only captures a rough row space description of $\mtx{A}$
but also is of rank at most $k$. A rank $k$ \ID~on $\mtx{C}$,
being an exact factorization of $\mtx{C}$ which is of
rank at most $k$, could hint on the relevant rows of
$\mtx{A}$ that approximate the entire row space of $\mtx{A}$ itself.
Specifically, similar to \eqref{eqn:column_approx_id2} where
approximating $\range(\mtx{A})$ using $\mtx{C}$ incurs an error term
$\begin{bmatrix}\mtx{0}&\mtx{Q}_2 S_{22}\end{bmatrix}$,
we can estimate the error of approximating $\range(\mtx{A}^*)$
using $\mtx{A}(I(1:k),:)$; see Lemma \ref{lem:Etilde} below.

The \CURID~algorithm is based on the two sided \ID~factorization,
and as a starting point, we assume the factorization \eqref{eqn:two_sided_approx_id}
has been computed using the procedures described in Section \ref{sec:prel}.
In other words, we assume that the index vectors $I$ and $J$,
and the basis matrices $\mtx{V}$ and $\mtx{W}$, are all available.
We then define
\begin{equation}
\label{eq:choice_C_and_R}
\mtx{C}
\ =\
\mtx{A}(:,J(1:k))
\quad\text{and}\quad
\mtx{R} = \mtx{A}(I(1:k),:).
\end{equation}
Consequently, $\mtx{C}$ and $\mtx{R}$ are respectively subsets of columns and of
rows of $\mtx{A}$, with $J$ and $I$ determined by the pivoted \QR~factorizations.
Next we construct a $k\times k$ matrix $\mtx{U}$ such that
$\mtx{A}\approx \mtx{C}\mtx{U}\mtx{R}$.
We know that
\begin{equation}
\label{eq:netflix1}
\mtx{A} \approx \mtx{C}\,\mtx{V}^{*},
\end{equation}
and we seek a factor $\mtx{U}$ such that
\begin{equation}
\label{eq:netflix2}
\mtx{A} \approx \mtx{C}\,\mtx{U}\,\mtx{R}.
\end{equation}
By inspecting (\ref{eq:netflix1}) and (\ref{eq:netflix2}), we find that we would
achieve our objective if we could determine a matrix $\mtx{U}$ such that
\begin{equation}
\label{eq:overd}
\begin{array}{cccccccccccccc}
\mtx{U} & \mtx{R} &=& \mtx{V}^{*}.\\
k\times k & k\times m && k\times m
\end{array}
\end{equation}
Unfortunately, (\ref{eq:overd}) is an over-determined system, but at least intuitively,
it seems plausible that it should have a fairly accurate solution, given that the
rows of $\mtx{R}$ and the rows of $\mtx{V}^{*}$ should, by construction, span roughly the
same space (namely, the space spanned by the $k$ leading right singular vectors
of $\mtx{A}$). Solving (\ref{eq:overd}) in the least-square sense, we arrive at our
definition of $\mtx{U}$:
\begin{equation}
\label{eq:defU}
\mtx{U} := \mtx{V}^* \mtx{R}^{\dagger}.
\end{equation}

The construction of $\mtx{C}$, $\mtx{U}$, and $\mtx{R}$ in the previous paragraph was based
on heuristics. We next demonstrate that the approximation error is comparable
to the error resulting from the original QR-factorization. First, let us
define $\mtx{E}$ and $\tilde{\mtx{E}}$ as the errors in the column and row IDs of $\mtx{A}$,
respectively,
\begin{align}
\label{eq:approx1}
\mtx{A} =&\ \mtx{C}\,\mtx{V}^{*} + \mtx{E},\\
\label{eq:approx2}
\mtx{A} =&\ \mtx{W}\,\mtx{R} + \tilde{\mtx{E}}.
\end{align}
Recall that $\mtx{E}$ is a quantity we can control by continuing the original
QR factorization until $\|\mtx{E}\|$ is smaller than some given threshold.
We will next prove two lemmas. The first states that the error in the CUR
decomposition is bounded by $\|\mtx{E}\| + \|\tilde{\mtx{E}}\|$. The second states
that $\|\tilde{\mtx{E}}\|$ is small whenever $\|\mtx{E}\|$ is small (and again,
$\|\mtx{E}\|$ we can control).

\begin{lemma}
\label{lem:CUR}
Let $\mtx{A}$ be an $m\times n$ matrix that satisfies the approximate
factorizations (\ref{eq:approx1}) and (\ref{eq:approx2}). Suppose further
that $\mtx{R}$ is full rank, and that the $k\times k$ matrix $\mtx{U}$ is defined by
(\ref{eq:defU}). Then
\begin{equation}
\label{eq:A_minus_CUR_twoEs}
\|\mtx{A} - \mtx{C}\mtx{U}\mtx{R}\| \leq \|\mtx{E}\| + \|\tilde{\mtx{E}}\|.
\end{equation}
\end{lemma}

\begin{proof}
Using first (\ref{eq:defU}) and then (\ref{eq:approx1}), we find
\begin{equation}
\label{eq:andare1}
\mtx{A} - \mtx{C}\mtx{U}\mtx{\mtx{R}} =
\mtx{A} - \mtx{C}\mtx{V}^{*}\mtx{\mtx{R}}^{\dagger}\mtx{\mtx{R}} =
\mtx{A} - (\mtx{A} - \mtx{E})\mtx{R}^{\dagger}\mtx{R} =
\bigl(\mtx{A} - \mtx{A}\mtx{R}^{\dagger}\mtx{R}\bigr) + \mtx{E}\mtx{R}^{\dagger}\mtx{R}.
\end{equation}
To bound the term $\mtx{A} - \mtx{A}\mtx{R}^{\dagger}\mtx{R}$ we use (\ref{eq:approx2}) and the fact that $\mtx{R}\mtx{R}^{\dagger}\mtx{R} = \mtx{R}$ to achieve
\begin{equation}
\label{eq:andare2}
\mtx{A} - \mtx{A}\mtx{R}^{\dagger}\mtx{R} =
\mtx{A} - (\mtx{W}\mtx{R} + \tilde{\mtx{E}})\mtx{R}^{\dagger}\mtx{R} =
\mtx{A} - \mtx{W}\mtx{R} - \tilde{\mtx{E}}\mtx{R}^{\dagger}\mtx{R} =
\tilde{\mtx{E}} - \tilde{\mtx{E}}\mtx{R}^{\dagger}\mtx{R} =
\tilde{\mtx{E}}(I-\mtx{R}^{\dagger}\mtx{R}).
\end{equation}
Inserting (\ref{eq:andare2}) into (\ref{eq:andare1}) and taking the norms of the result, we get
$$
\|\mtx{A} - \mtx{C}\mtx{U}\mtx{R}\| =
\|\tilde{\mtx{E}}(\mtx{I}-\mtx{R}^{\dagger}\mtx{R}) + \mtx{E}\mtx{R}^{\dagger}\mtx{R}\| \leq
\|\tilde{\mtx{E}}(\mtx{I}-\mtx{R}^{\dagger}\mtx{R})\| + \|\mtx{E}\mtx{R}^{\dagger}\mtx{R}\| \leq
\|\tilde{\mtx{E}}\| + \|\mtx{E}\|,
$$
where in the last step we used that $\mtx{R}\mtx{R}^{\dagger}$ and
$\mtx{I} - \mtx{R}\mtx{R}^{\dagger}$ are both orthonormal projections.
\end{proof}

\begin{lemma}
\label{lem:Etilde}
Let $\mtx{A}$ be an $m\times n$ matrix that admits the factorization (\ref{eq:approx1}),
with error term $\mtx{E}$.
Suppose further that $I = [I_{\rm skel},I_{\rm res}]$ and $\mtx{T}$ form
the output of the ID of the matrix $\mtx{C}$, so that
\begin{equation}
\label{eq:hannah1}
\mtx{C} = \mtx{W}\mtx{C}(I_{\rm skel},:),\qquad\mbox{where}\qquad \mtx{W} = \mtx{P}\left[\begin{array}{c}\mtx{I} \\ \mtx{T}^{*}\end{array}\right],
\end{equation}
and where $\mtx{P}$ is the permutation matrix for which
$\mtx{P}\mtx{A}(I,:) = \mtx{A}$. Now define the matrix $\mtx{R}$ via
\begin{equation}
\label{eq:hannah2}
\mtx{R} = \mtx{A}(I_{\rm skel},:).
\end{equation}
Observe that $\mtx{R}$ consists of the $k$ rows of $\mtx{A}$ selected in the skeletonization of $\mtx{C}$.
Finally, set
\begin{equation}
\label{eq:hannah3}
\mtx{F} = \bigl[-\mtx{T}^{*}\ \ \mtx{I}\bigr]\mtx{P}^{*}.
\end{equation}
Then the product $\mtx{W}\mtx{R}$ approximates $\mtx{A}$, with a residual error
\begin{equation}
\label{eq:hannah4}
\tilde{\mtx{E}} = \mtx{A} - \mtx{W}\mtx{R} = \mtx{P}\left[\begin{array}{c} \mtx{0} \\ \mtx{F}\mtx{E}\end{array}\right].
\end{equation}
\end{lemma}

\begin{proof}
From the definitions of $\mtx{W}$ in (\ref{eq:hannah1}) and $\mtx{R}$ in
(\ref{eq:hannah2}) we find
\begin{multline}
\label{eq:elec1}
\mtx{A} - \mtx{W}\mtx{R} = \mtx{P} \mtx{A}(I,:) - \mtx{W}\mtx{R} =
\mtx{P}\left[\begin{array}{c} \mtx{A}(I_{\rm skel},:) \\ \mtx{A}(I_{\rm res},:) \end{array}\right] -
\mtx{P}\left[\begin{array}{c} \mtx{I} \\ \mtx{T}^{*} \end{array}\right]\mtx{A}(I_{\rm skel},:)\\
= \mtx{P}\left[\begin{array}{c} \mtx{0} \\ \mtx{A}(I_{\rm res},:) - \mtx{T}^{*}\mtx{A}(I_{\rm skel},:)\end{array}\right] =
\mtx{P}\left[\begin{array}{c} \mtx{0} \\ \mtx{F}\mtx{A}\end{array}\right].
\end{multline}
To bound the term $\mtx{F}\mtx{A}$ in (\ref{eq:elec1}), we invoke (\ref{eq:approx1})
to obtain
\begin{equation}
\label{eq:elec2}
\mtx{F}\mtx{A} = \mtx{F}\mtx{C}\mtx{V}^{*} + \mtx{F}\mtx{E} = \{\mbox{Insert (\ref{eq:hannah1})}\} = \mtx{F}\mtx{W}\mtx{C}(I_{\rm skel},:)\mtx{V}^{*} + \mtx{F}\mtx{E} = \mtx{F}\mtx{E},
\end{equation}
since $\mtx{F}\mtx{W} = \mtx{0}$ due to (\ref{eq:hannah1}) and (\ref{eq:hannah3}).
Finally, insert (\ref{eq:elec2}) into (\ref{eq:elec1}) to obtain (\ref{eq:hannah4}).
\end{proof}

Equation (\ref{eq:hannah4}) allows us to bound the norm of the error $\tilde{\mtx{E}}$
in (\ref{eq:approx2}). Simply observe that the definition of $\mtx{F}$ in (\ref{eq:hannah3})
implies that for any matrix $\mtx{X}$ we have:
\begin{equation*}
\mtx{F} \mtx{X} = \bigl[-\mtx{T}^{*}\ \ \mtx{I}\bigr]\mtx{P}^{*} \mtx{X} = \bigl[-\mtx{T}^{*}\ \ \mtx{I}\bigr] \left[\begin{matrix} \mtx{X}(I_{\rm skel},:) \\ \mtx{X}(I_{\rm res},:) \end{matrix} \right] = -\mtx{T}^* \mtx{X}(I_{\rm skel},:) + \mtx{X}(I_{\rm res},:),
\end{equation*}
so that:
\begin{equation}
\label{eq:error_bound_FX}
\|\mtx{F}\mtx{X}\| = \|\mtx{X}(I_{\rm res},:) - \mtx{T}^*\mtx{X}(I_{\rm skel},:)\| \leq
\|\mtx{X}(I_{\rm res},:)\| + \|\mtx{T}\|\,\|\mtx{X}(I_{\rm skel},:)\| \leq
(1 + \|\mtx{T}\|)\,\|\mtx{X}\|.
\end{equation}
This leads us to the following Corollary to Lemma \ref{lem:Etilde}:

\newpage
\begin{corollary}
Under the same assumptions as in Lemma \ref{lem:Etilde}, we have
\begin{equation}
\label{eq:error_bound_tildeE}
\|\tilde{\mtx{E}}\| \leq (1 + \|\mtx{T}\|)\,\|\mtx{E}\|.
\end{equation}
Further, assuming additionally that the conditions of Lemma \ref{lem:CUR} are satisfied,
\begin{equation}
\label{eq:error_bound_CUR}
\|\mtx{A} - \mtx{C}\mtx{U}\mtx{R}\| \leq (2 + \|\mtx{T}\|)\,\|\mtx{E}\|.
\end{equation}
\end{corollary}
\begin{proof}
To show \eqref{eq:error_bound_tildeE}, we use \eqref{eq:hannah4} and 
\eqref{eq:error_bound_FX}:
\begin{equation*}
\|\tilde{\mtx{E}}\| = \left\|\mtx{P} \begin{bmatrix}\mtx{0} \\ \mtx{F}\mtx{E}\end{bmatrix}\right\| \leq \left\|\begin{bmatrix}\mtx{0} \\ \mtx{F}\mtx{E}\end{bmatrix}\right\| \leq 
(1 + \|\mtx{T}\|)\,\|\mtx{E}\|. 
\end{equation*}
For \eqref{eq:error_bound_CUR}, we use \eqref{eq:A_minus_CUR_twoEs} and \eqref{eq:error_bound_tildeE}:
\begin{equation*}
\|\mtx{A} - \mtx{C}\mtx{U}\mtx{R}\| \leq \|\mtx{E}\| + \|\tilde{\mtx{E}}\| \leq (2 + \|\mtx{T}\|)\,\|\mtx{E}\|.
\end{equation*}
\end{proof}

Now recall that the matrix $\mtx{T}$ contains the expansion coefficients in the
interpolative decomposition of $\mtx{C}$. These can be guaranteed
\cite{Liberty18122007} to all be bounded
by $1 + \nu$ in magnitude for any positive number $\nu$. The cost increases as $\nu \rightarrow 0$,
but for, e.g., $\nu=1$, the cost is very modest. Consequently, we find that for either
the spectral or the Frobenius norm, we can easily
guarantee $\|\mtx{T}\| \leq (1+\nu)\sqrt{k(n-k)}$,
with practical norm often far smaller.

\section{Efficient deterministic algorithms}
\label{sec:detalg}

Sections \ref{sec:prel} and \ref{sect:CUR_ID} describe how to obtain the \ID,
two-sided \ID, and the \CUR~decompositions from the output of the column
pivoted rank $k$ \QR~algorithm. In this section, we discuss implementation
details, and computational costs for each of the three algorithms.

\subsection{The one-sided interpolative decomposition}
\label{sec:onesideID}

We start discussing the algorithm for computing an \ID~decomposition which
returns an index vector $J$ and a matrix $\mtx{V}$ such that
$\mtx{A} \approx \mtx{A}(:,J(1:k)) \mtx{V}^{*}$, and is summarized as
Algorithm \ref{algo:rank_k_id}.
The only computational complication here is how to evaluate
$\mtx{T} = \mtx{S}_{11}^{-1} \mtx{S}_{12}$ on Line 4 of the algorithm. Observe
that $\mtx{S}_{11}$ is upper triangular, so as long as $\mtx{S}_{11}$ is
not too ill-conditioned, a simple backwards solve will compute
$\mtx{T}$ very efficiently.
%For large matrices, one would likely not want to form the inverse matrix of
%$\mtx{S}_{11}$. Instead, recognizing that $\mtx{S}_{11}$ is upper
%triangular, we can solve for a column of matrix $T$ at a time. For each column $j$ from
%$1$ to $n-k$, we solve the system $\mtx{S}_{11} T_j = \mtx{S}_{12}(:,j)$, which is easy to
%do by back substitution.
When highly accurate factorizations are sought, however, $\mtx{S}_{11}$
will typically be sufficiently ill-conditioned that it is better
to view $\mtx{T}$ as the solution to a least squares system:
\begin{equation}
\label{eq:lsqr_for_T}
\mtx{T} = \arg\min_{\mtx{U}} \| \mtx{S}_{11}\mtx{U} - \mtx{S}_{12}\|.
\end{equation}
This equation can be solved using stabilized methods. For instance, we
can form a stabilized pseudo-inverse of $\mtx{S}_{11}$ by first computing its
SVD $\mtx{S}_{11} = \tilde{\mtx{U}} \tilde{\mtx{D}} \tilde{\mtx{V}}^*$.
Dropping all terms involving singular values smaller than some specified threshold,
we obtain a truncated decomposition
$\mtx{S}_{11} \approx \hat{\mtx{U}} \hat{\mtx{D}} \hat{\mtx{V}}^*$. Then set
$\mtx{T} = \hat{\mtx{V}} \hat{\mtx{D}}^{-1} \hat{\mtx{U}}^{*} \mtx{S}_{12}$.
We can also amend \eqref{eq:lsqr_for_T} with a regularization term (i.e. 
$\lambda \|\mtx{U}\|$), turning the minimization into a Tikhonov type problem, 
solvable by an application of the conjugate gradient scheme. 

There exists a variation of Algorithm \ref{algo:rank_k_id} that results in
an interpolation matrix $\mtx{V}$ whose entries are \textit{assured} to be
of moderate magnitude. The idea is to replace the column pivoted QR on
Line 1 by the so called ``strongly rank revealing QR factorization'' algorithm
described by Gu and Eisenstat in \cite{GuEisenstat1}. They prove that for any 
$\epsilon > 0$, one can construct matrices $\mtx{S}_{11}$ and $\mtx{S}_{12}$ such that the
equation $\mtx{S}_{11}\mtx{T} = \mtx{S}_{12}$ has a solution for which
$|\mtx{T}(i,j)| \leq 1 + \epsilon$ for every $i$ and $j$. The cost of the
algorithm increases as $\epsilon \rightarrow 0$, but remains reasonable
as long as $\epsilon$ is not too close to $0$. While such a provably robust
algorithm has strong appeal, we have found that in practice, standard column
pivoted \QR~works so well that the additional cost and coding effort required
to implement the method of \cite{GuEisenstat1} is not worthwhile.

With respect to storage cost, if $\mtx{A}$ is $m \times n$, to store the 
\ID~representation of $\mtx{A}$, 
we require $mk + k(n-k)$ units (since $\mtx{V}$ contains within it an 
identity matrix).

\subsection{The two-sided interpolative decomposition}

Next, we consider the two-sided \ID~described in Section \ref{sec:twoID}, and
summarized here as Algorithm \ref{algo:rank_k_tsid}.
The main observation is that $\mtx{C}^*$ is a matrix of rank at most $k$. Hence, a
rank $k$ \QR~decomposition would reconstruct it exactly so that the steps in
Algorithm \ref{algo:rank_k_id} produce an exact decomposition. Typically,
if the dimensions are not too large, the \QR~decomposition for step 2 can
be performed using standard software packages, such as, e.g., LAPACK.
For the two sided \ID, the storage requirement for an $m \times n$ matrix 
is $k(m-k) + k^2 + k(n-k)$, which is the same as for the one sided \ID~above.

\subsection{The CUR decomposition}

As demonstrated in Section \ref{sect:CUR_ID}, it is simple to convert Algorithm \ref{algo:rank_k_tsid}
for computing a two-sided ID into an algorithm for constructing the \CUR~decomposition.
We summarize the procedure as Algorithm \ref{algo:rank_k_cur}. The only complication here
concerns solving the least squares problem
\begin{equation}
\label{eq:cur_ur_sys}
\begin{array}{cccccccccc}
\mtx{U} & \mtx{R} &=& \mtx{V}^*\\
k\times k & k\times n && k\times n
\end{array}
\end{equation}
for $\mtx{U}$. In applications like data-mining, where $n$ might be very large, and modest accuracy
is sought, one may simply form the normal equations and solve those. For higher accuracy,
stabilized techniques based on a truncated QR or SVD decomposition of $\mtx{R}$ is preferable.
%In cases when $\mtx{R}$ is sufficiently small, we can construct
%$R^{\dagger}$ via the SVD of $\mtx{R}$. Otherwise, one may solve the least squares
%problem for the matrix $H = U^T$ from the linear system $R^T H = V$, via the quadratic
%equations $(R R^T) H = R V$.

If feasible, one may also consider some adjustment to
\eqref{eq:cur_ur_sys} based on the error introduced by the truncated \QR~factorization.
Including the error term from \eqref{eq:id_error}, we may write:
\begin{equation*}
\mtx{A} = \mtx{C} \mtx{V}^* + \mtx{E} = \mtx{C}\mtx{U}\mtx{R},
\end{equation*}
from which we obtain the modified system:
\begin{equation}
\label{eq:cur_ur_sys2}
\mtx{U}\mtx{R} = \mtx{V}^* + \mtx{C}^{\dagger} \mtx{E},
\end{equation}
where $\mtx{E}$ can be obtained from $\mtx{E} = \mtx{A} - \mtx{Q}\mtx{R}$ once the
partial rank $k$ \QR~factorization has been performed. One can then obtain matrix
$\mtx{U}$ from a least squares problem
corresponding to \eqref{eq:cur_ur_sys2}.
For \CUR, the storage requirement for an $m\times n$ matrix is $mk + kn + k^2$, 
noting that the $k \times k$ matrix $\mtx{U}$ is not a diagonal.

\subsection{Computational and storage costs}

All the algorithms discussed in this section have asymptotic cost $O(mnk)$.
The dominant part of the computation is almost always the initial rank-$k$
QR factorization. All subsequent computations involve only matrices of sizes
$m\times k$ or $k\times n$, and have cost $O((m+n)k^{2})$.
In terms of memory storage, when the matrix $\mtx{A}$ is dense, the two \ID~decompositions
of $\mtx{A}$ require the least space, followed by the \SVD, and then the \CUR.
However, if $\mtx{A}$ is a sparse matrix and sparse storage format is used
for the factor matrices, the \ID~and \CUR~decompositions can be stored more efficiently. 
Note that the factors $\mtx{C}$ and $\mtx{R}$ will be sparse if $\mtx{A}$ is sparse 
and so in the sparse case, the \CUR~storage will in general be minimal amongst all the 
factorizations.

\begin{figure*}[!ht]
\centering
\begin{minipage}[c]{14cm}
\begin{algorithm}[H]
\SetKwInOut{Input}{Input}
\SetKwInOut{Output}{Output}
%\SetKwBlock{Begin}{loop}{end loop}
\caption{A rank $k$ \ID~decomposition \label{algo:rank_k_id}}
%\Input{\mtx{A}n $m\times n$ matrix $\mtx{\mtx{A}}$ and parameter $k < \min(m,n)$.}
\Input{$\mtx{A}\in\mathbb{C}^{m\times n}$ and parameter $k < \min(m,n)$.}
\Output{\mtx{A} column index set $J$ and a matrix
$\mtx{V}\in\mathbb{C}^{n\times k}$
such that $\mtx{A} \approx \mtx{A}(:,J(1:k)) \mtx{V}^{*}$.}
\BlankLine
Perform a rank $k$ column pivoted \QR~factorization to get
$\mtx{A} \mtx{P} = \mtx{Q}_1 \mtx{S}_1$\;
%\nonl
\Indp
%\texttt{\% i.e., $P\in\mathbb{R}^{n\times n}$ is a permutation matrix,
%$\mtx{Q}_1\in\mathbb{C}^{m\times k}$ has orthonormal columns and
%$\mtx{S}_1\in\mathbb{C}^{k\times n}$ is upper triangular}
\Indm define the ordered index set $J$ via $\mtx{I}(:,J) = \mtx{P}$\;
partition $\mtx{S}_1$: $\mtx{S}_{11} = \mtx{S}_1(:,1:k)$,
$\mtx{S}_{12} = \mtx{S}_1(:,k+1:n)$\;
$\mtx{V} =
\mtx{P} \begin{bmatrix} \mtx{I}_k & \mtx{S}_{11}^{-1} \mtx{S}_{12} \end{bmatrix}^*$\;
\end{algorithm}
\end{minipage}
\end{figure*}

\begin{figure*}[!ht]
\centering
\begin{minipage}[c]{14cm}
\begin{algorithm}[H]
\SetKwInOut{Input}{Input}
\SetKwInOut{Output}{Output}
%\SetKwBlock{Begin}{loop}{end loop}
\caption{A rank $k$ two sided \ID~decomposition \label{algo:rank_k_tsid}}
\Input{$\mtx{A}\in\mathbb{C}^{m\times n}$ and parameter $k < \min(m,n)$.}
\Output{A column index set $J$, a row index set $I$ and a matrices
$\mtx{V}\in\mathbb{C}^{n\times k}$ and $\mtx{W}\in\mathbb{C}^{m \times k}$
such that $\mtx{A} \approx \mtx{W} \mtx{A}(I(1:k),J(1:k)) \mtx{V}^{*}$.}
\BlankLine
Perform a one sided rank $k$ \ID~of $\mtx{A}$ so that
$\mtx{A} \approx \mtx{C} \mtx{V}^{*}$ where $\mtx{C} = \mtx{A}(:,J(1:k))$\;
Perform a full rank \ID~on $\mtx{C}^*$ so that
$\mtx{C}^{*} = \mtx{C}^*(:,I(1:k)) \mtx{W}^{*}$\;
\end{algorithm}
\end{minipage}
\end{figure*}

\begin{figure*}[!ht]
\centering
\begin{minipage}[c]{14cm}
\begin{algorithm}[H]
\SetKwInOut{Input}{Input}
\SetKwInOut{Output}{Output}
%\SetKwBlock{Begin}{loop}{end loop}
\caption{A rank $k$ \CURID~algorithm \label{algo:rank_k_cur}}
\Input{$\mtx{A}\in\mathbb{C}^{m\times n}$ and parameter $k < \min(m,n)$.}
\Output{Matrices $\mtx{C} \in \mathbb{C}^{m \times k}$, $\mtx{R} \in \mathbb{C}^{k \times n}$, and
$\mtx{U} \in \mathbb{C}^{k \times k}$ (such that $\mtx{A} \approx \mtx{C} \mtx{U} \mtx{R}$).}
\BlankLine
Construct a rank $k$ two sided \ID~of $\mtx{A}$ so that
$\mtx{A} \approx \mtx{W} \mtx{A}(I(1:k),J(1:k)) \mtx{V}^{*}$\;
Construct matrices $\mtx{C} = \mtx{A}(:,J(1:k))$ and $\mtx{R} = \mtx{A}(I(1:k),:)$\;
Construct matrix $\mtx{U}$ via $\mtx{U} = \mtx{V}^{*} \mtx{R}^{\dagger}$\;
\end{algorithm}
\end{minipage}
\end{figure*}

\section{Efficient randomized algorithms}
\label{sec:randalg}

The computational costs of the algorithms described in Section \ref{sec:detalg}
tend to be dominated by the cost of performing the initial $k$ steps of a
column pivoted QR-decomposition of $\mtx{A}$ (at least when the rank $k$ is substantially
smaller than the dimensions $m$ and $n$ of the matrix). This initial step can often
be accelerated substantially by exploiting techniques based on randomized projections.
These ideas were originally proposed in \cite{2006_martinsson_random1_orig,sarlos2006improved}, and
further developed in \cite{2009_szlam_power,2007_woolfe_liberty_rokhlin_tygert,Liberty18122007,MR2806637}.

Observe that in order to compute the column \ID~of a matrix, all we need is to
know the linear dependencies among the columns of $\mtx{A}$. When the singular values
of $\mtx{A}$ decay reasonably rapidly, we can determine these linear dependencies by
processing a matrix $\mtx{Y}$ of size $\ell \times n$, where $\ell$ can be much smaller
than $n$. The rows of $\mtx{Y}$ consist of random linear combinations of
the rows of $\mtx{A}$,
and as long as the number of samples $\ell$ is a ``little bit'' larger than the
rank $k$, highly accurate approximations result. In this section, we provide a
brief description of how randomization can be used to accelerate the ID and the CUR
factorizations, for details and a rigorous analysis of sampling errors, see
\cite{MR2806637}.

The techniques in this section are all designed to compute a one-sided \ID.
Once this factorization is available, either a two-sided \ID, or a
\CUR~decomposition can easily be obtained using the techniques outlined
in Section \ref{sect:CUR_ID}.

\subsection{A basic randomized algorithm}
\label{sec:randbasic}
Suppose that we are given an $m\times n$ matrix $\mtx{A}$ and seek to compute
a column ID, a two-sided ID, or a CUR decomposition. As we saw in Section
\ref{sec:detalg}, we can perform this task as long as we can identify an
index vector $J = [J_{\rm skel},\,J_{\rm res}]$ and a basis matrix
$\mtx{V} \in \mathbb{C}^{n\times k}$ such that
$$
\begin{array}{cccccccccc}
\mtx{A} &=& \mtx{A}(:,J_{\rm skel}) & \mtx{V}^{*} &+& \mtx{E}\\
m\times n && m\times k & k\times n && m\times n
\end{array}
$$
where $\mtx{E}$ is small. In Section \ref{sec:detalg}, we found $J$ and $\mtx{V}$
by performing a column pivoted QR factorization of $\mtx{A}$. In order to do this via
randomized sampling, we first fix a small over-sampling parameter $p$, say $p=10$
for now (see Remark \ref{remark:p} for details). Then draw a $(k+p)\times m$
random matrix $\mtx{\Omega}$ whose entries are i.i.d.~standardized Gaussian random
variables, and form the \textit{sampling matrix}
\begin{equation}
\label{eq:Y=OmegaA}
\begin{array}{cccccccc}
\mtx{Y} &=& \mtx{\Omega} & \mtx{A}.\\
(k+p)\times n && (k+p)\times m & m\times n
\end{array}
\end{equation}
One can prove that with high probability, the space spanned by the rows of $\mtx{Y}$
contains the dominant $k$ right singular vectors of $\mtx{A}$ to high accuracy. This
is precisely the property we need in order to find both the vector $J$ and the
basis matrix $\mtx{V}$. All we need to do is to perform $k$ steps of a column pivoted
QR factorization of the sample matrix to form a partial QR factorization
$$
\begin{array}{cccccccc}
\mtx{Y}(:,J) &\approx& \mtx{Q} & \mtx{S}.\\
(k+p)\times n && (k+p)\times k & k\times n
\end{array}
$$
Then compute the matrix of expansion coefficients via $\mtx{T} = \mtx{S}(1:k,1:k)^{-1}\mtx{S}(1:k,(k+1):n)$,
or a stabilized version, as described in Section \ref{sec:onesideID}.
%Then set
%$$
%J_{\rm skel} = J(1:k),\qquad
%J_{\rm res}  = J((k+1):n),\qquad\mbox{and}\qquad
%\mtx{T}            = \mtx{R}(1:k,1:k)\backslash \mtx{R}(1:k,(k+1):n).
%$$
The matrix $\mtx{V}$ is formed from $\mtx{T}$ as before,
resulting in Algorithm \ref{algo:rank_k_rand_id}.
The asymptotic cost of Algorithm \ref{algo:rank_k_rand_id} is $O(mnk)$,
just like the algorithms described in Section \ref{sec:detalg}.
However, substantial practical gain is achieved due to the fact that the
matrix-matrix multiplication is much faster than a column-pivoted QR factorization.
This effect gets particularly pronounced when a matrix is very large and is
stored either out-of-core, or on a distributed memory machine.

\begin{figure*}
\centering
\begin{minipage}[c]{14cm}
\begin{algorithm}[H]
\SetKwInOut{Input}{Input}
\SetKwInOut{Output}{Output}
%\SetKwBlock{Begin}{loop}{end loop}
\caption{A randomized rank $k$ \ID~Decomposition \label{algo:rank_k_rand_id}}
%\Input{An $m\times n$ matrix $\mtx{A}$ and parameter $k < \min(m,n)$.}
\Input{$\mtx{A}\in\mathbb{C}^{m\times n}$, a rank parameter $k < \min(m,n)$,
and an oversampling parameter $p$.}
\Output{A column index set $J$ and a matrix
$\mtx{V}\in\mathbb{C}^{n\times k}$
(such that $\mtx{A} \approx \mtx{A}(:,J(1:k)) \mtx{V}^{*}$).}
\BlankLine
Construct a random matrix $\mtx{\Omega} \in \mathbb{R}^{(k+p) \times m}$
with i.i.d.~Gaussian entries\;
Construct the sample matrix $\mtx{Y} = \mtx{\Omega} \mtx{A}$\;
%\eIf{$l = 0$}{
%    Construct the sample matrix $Y = \Omega A$\;
%    }{
%      Set $Y = \Omega A$\;
%      \For{$i = 1$ \textbf{to} $l$}{
%        Set $Y = orth(Y) A^*$\;
%        Set $Y = orth(Y) A$\;
%      }
%}
Perform full pivoted \QR~factorization on $\mtx{Y}$ to get:
$\mtx{Y} \mtx{P} = \mtx{Q} \mtx{S}$\;
Remove $p$ columns of $\mtx{Q}$ and $p$ rows of $\mtx{S}$ to construct
$\mtx{Q}_1$ and $\mtx{S}_1$\;
Define the ordered index set $J$ via $\mtx{I}(:,J) = \mtx{P}$\;
Partition $\mtx{S}_1$: $\mtx{S}_{11} = \mtx{S}_1(:,1:k)$,
$\mtx{S}_{12} = \mtx{S}_1(:,k+1:n)$\;
$\mtx{V} = \mtx{P} \begin{bmatrix} \mtx{I}_k & \mtx{S}_{11}^{-1} \mtx{S}_{12} \end{bmatrix}^*$\;
\end{algorithm}
\end{minipage}
\end{figure*}

\begin{remark}
\label{remark:p}
Careful mathematical analysis is available to guide the choice of
the over-sampling parameter $p$ \cite{MR2806637}. However, in practical
applications, choosing $p=10$ is almost always more than sufficient. If a very close
to optimal skeleton is desired, one could increase the parameter up to
$p=2k$, but this is generally far higher than needed.
\end{remark}

\subsection{An accelerated randomized scheme}
\label{sec:SRFT}

At this point, all algorithms described have asymptotic complexity $O(mnk)$.
Using the randomized projection techniques, we can reduce this to $O(mn\,\log(k) + k^{2}(m+n))$. The idea is to replace the Gaussian randomized matrix $\mtx{\Omega}$ we used in Section
\ref{sec:randbasic} by a random matrix that has enough structure that the matrix-matrix
multiplication (\ref{eq:Y=OmegaA}) can be executed in $O(mn\,\log(k))$ operations.
For instance, one can use a \textit{subsampled random Fourier transform (SRFT)},
which takes the form
\begin{equation}
\label{eq:def_srft}
\begin{array}{cccccccccccccccc}
\mtx{\Omega} &=& \sqrt{\frac{m}{\ell}} & \mtx{R} & \mtx{F} & \mtx{D}\\
\ell \times m &&& \ell \times m & m\times m & m\times m
\end{array}
\end{equation}
where $\mtx{D}$ is an $m \times m$ diagonal matrix whose entries are
independent random variables uniformly distributed on the complex unit circle;
where $\mtx{F}$ is the $m \times m$ unitary discrete Fourier transform, whose
entries take the values $\mtx{F}(p,q) = m^{-1/2} \, e^{-2\pi i (p-1)(q-1)/m}$
for $p, q = 1, 2, \dots, m$;
and where $\mtx{R}$ is an $\ell \times m$ matrix that samples $\ell$ coordinates
from $m$ uniformly at random (i.e., its $\ell$ rows are drawn randomly
without replacement from the rows of the $m \times m$ identity matrix).

When using an SRFT, a larger number of samples is sometimes required to
attain similar accuracy. In practice $\ell = 2k$ is almost always sufficient,
see \cite[Sec.~4.6]{MR2806637}.

Replacing lines 1 and 2 in Algorithm \ref{algo:rank_k_rand_id} by the SRFT
(\ref{eq:def_srft}) reduces the cost of executing these lines to $O(mn\,\log(k))$,
assuming $\ell = 2k$. The remaining operations have complexity $O(k^{2}(m+n))$.

\subsection{An accuracy enhanced scheme}

The randomized sampling schemes described in Sections \ref{sec:randbasic} and
\ref{sec:SRFT} are roughly speaking as accurate as the techniques based on
a column pivoted QR factorization described in Section \ref{sec:detalg} as
long as the singular values of $\mtx{A}$ exhibit reasonable decay. For the case
where the singular values decay slowly (as often happens in data mining and
analysis of statistical data, for instance), the accuracy deteriorates. However,
high accuracy can easily be restored by slightly modifying the construction of
the sampling matrix $\mtx{Y}$. The idea of the power sampling scheme is roughly 
to choose a small integer $q$ (say $q=1$ or $q=2$), and then form the sampling matrix via
\begin{equation}
\label{eq:Ypower}
\mtx{Y} = \mtx{\Omega}\,\mtx{A}\,\bigl(\mtx{A}^{*}\mtx{A})^{q}.
\end{equation}
The point here is that if $\mtx{A}$ has singular values $\{\sigma_{j}\}_{j=1}^{\min(m,n)}$,
then the singular values of $\mtx{A}\,\bigl(\mtx{A}\mtx{A}^{*})^{q}$
are $\{\sigma_{j}^{2q+1}\}_{j=1}^{\min(m,n)}$,
which means that the larger singular values are weighted much more heavily versus
the lower ones.

For computational efficiency, note that the evaluation of (\ref{eq:Ypower}) should be
done by successive multiplications of $\mtx{A}$ and $\mtx{A}^{*}$, so that line 2 in Algorithm
\ref{algo:rank_k_rand_id} gets replaced by:
\begin{center}
\begin{minipage}{100mm}
\begin{tabbing}
\= \hspace{10mm} \= \hspace{5mm} \= \hspace{5mm} \= \hspace{5mm} \\ \kill
\> (2a) \> $\mtx{Y} = \mtx{\Omega} \mtx{A}$\\
\> (2b) \> \textbf{for} $i = 1:q$\\
\> (2c) \> \> $\mtx{Y} \leftarrow \mtx{Y} \mtx{A}^*$\\
\> (2d) \> \> $\mtx{Y} \leftarrow \mtx{Y} \mtx{A}$\\
\> (2e) \> \textbf{end}
\end{tabbing}
\end{minipage}
\end{center}
In cases where very high computational precision is required (higher than
$\epsilon_{\rm mach}^{1/(2q+1)}$, where $\epsilon_{\rm mach}$ is the machine
precision), one typically needs to orthonormalize the sampling matrix in
between multiplications, resulting in:
\begin{center}
\begin{minipage}{100mm}
\begin{tabbing}
\= \hspace{10mm} \= \hspace{5mm} \= \hspace{5mm} \= \hspace{5mm} \\ \kill
\> (2a) \> $\mtx{Y} = \mtx{\Omega} \mtx{A}$\\
\> (2b) \> \textbf{for} $i = 1:q$\\
\> (2c) \> \> $\mtx{Y} \leftarrow \texttt{orth}(\mtx{Y}) \mtx{A}^*$\\
\> (2d) \> \> $\mtx{Y} \leftarrow \texttt{orth}(\mtx{Y}) \mtx{A}$\\
\> (2e) \> \textbf{end}
\end{tabbing}
\end{minipage}
\end{center}
where $\texttt{orth}$ refers to orthonormalization of the \textit{rows},
without pivoting. In other words, if $\mtx{Q} = \texttt{orth}(\mtx{Y})$, then $\mtx{Q}$
is a matrix whose rows form an orthonormal basis for the rows of $\mtx{Y}$.

The asymptotic cost of the algorithm described in this section is
$O((2q+1)mnk + k^{2}(m+n))$.

\begin{remark}
It is to the best of our knowledge not possible to accelerate the accuracy
enhanced technique described in this section to $O(mn\,\log(k))$ complexity.
\end{remark}

%
%\begin{algorithm}[!ht]
%\SetKwInOut{Input}{Input}
%\SetKwInOut{Output}{Output}
%%\SetKwBlock{Begin}{loop}{end loop}
%\caption{Randomized rank $k$ ID
%\label{algo:randid}}
%\Input{An $m\times n$ matrix $\mtx{A}$, an integer rank parameter $0 < k < \min(m,n)$, and an
%oversampling integer parameter $p$}
%\Output{Index set $J \in \mathbb{Z}_{+}^{n}$ and matrix $T \in \mathbb{R}^{k \times (n-k)}$.}
%\BlankLine
%Form $(k + p) \times m$ Gaussian random matrix $\Omega$. \;
%Form $(k + p) \times n$  matrix $Y = \Omega A$. \;
%Perform $k$ steps of pivoted \QR~decomposition on $Y$
%\begin{equation*}
%Y(:,J) \approx \tilde{Q}_k \tilde{S}_k
%\end{equation*} \;
%Remove $p$ columns of $\tilde{Q}_k$ and $p$ rows of $\tilde{S}_k$ to form
%$Q$ and $S$ such that:
%\begin{equation*}
%Y(:,J) \ \approx \
%\kbordermatrix{&
%k \\
%k+p &Q}
%\times
%\kbordermatrix{&
%k & n-k\\
%k &\mtx{S}_{11} & \mtx{S}_{12}
%}.
%%\ = \
%%\kbordermatrix{&
%%k & n-k\\
%%m &\mtx{Q}_1 \mtx{S}_{11} & \mtx{Q}_1 \mtx{S}_{12} + \mtx{Q}_2 S_{22}
%%}.
%\end{equation*}
%\;
%Solve the linear system \;
%\end{algorithm}
%

\newpage
\section{Numerics}
\label{sec:num}
In this section, we present numerical comparisons between the proposed \CURID~algorithm,
and previously proposed schemes, specifically those implemented in the rCUR 
package \cite{rCUR} and the algorithm from \cite{2014arXiv1407.5516S}.

We first compare the proposed method for computing the \CUR~decomposition
(Algorithm \ref{algo:rank_k_cur}) against four existing \CUR~algorithms,
one based on the newly proposed DEIM-CUR method as described in \cite{2014arXiv1407.5516S}
and three algorithms as implemented in the rCUR package. We first use the full SVD
with each algorithm:
\begin{enumerate}
\item[\CUR-H] The full \SVD~is computed and provided to rCUR, and then 
the ``highest ranks'' option is chosen. This generally offers good performance 
and reasonable runtime in our experiments.
\item[\CUR-1] The full \SVD~is computed and provided to rCUR, and then the
``orthogonal top scores'' option is chosen. This is an expensive scheme
that we believe gives the best performance in rCUR for many 
matrix types. However, when the decay of singular values of the input matrix 
is very rapid or abrupt (as in the example in Figure \ref{fig:setIV} below),
the scheme performs poorly. This scheme is also considerably slower than the others. 
\item[\CUR-2] The full \SVD~is computed and provided to DEIM-CUR. 
This generally offers good performance and reasonable runtime in our experiments.
\item[\CUR-3] The full \SVD~is computed and provided to rCUR, and then the
``top scores'' option is chosen. This procedure reflects a common way that 
``leverage scores'' are used. It has slightly worse performance than \CUR-1 
and \CUR-H in our experiments but better runtime.
%we also use schemes computed using the randomized \SVD~in \cite{MR2806637}.
\end{enumerate}
Our first set of test matrices (``Set 1'') involves 
matrices $\mtx{A}$ of size $1000\times 3000$,
of the form $\mtx{A} = \mtx{U}\,\mtx{D}\,\mtx{V}^{*}$ where $\mtx{U}$ and $\mtx{V}$
are random orthonormal matrices, and $\mtx{D}$ is a
diagonal matrix with entries that are logspaced between $1$ and
$10^{b}$, for $b = -2,\,-4,\,-6$.
The second set (``Set 2'') are simply the transposes of the matrices in Set 1 (so these are
matrices of size $3000\times 1000$).
Figure \ref{fig:setI_and_II} plots the median relative errors in the spectral 
norm between the matrix
$\mtx{A}$ and the corresponding factorization 
(with the error defined as 
$E = \frac{\|\hat{\mtx{A}}_k - \mtx{A}\|}{\|\mtx{A}\|}$ where 
$\hat{\mtx{A}}_k = \mtx{C} \mtx{U} \mtx{R}$ is the corresponding 
approximation of given rank). We plot median quantities 
collected over $5$ trials. In addition to the four \CUR~algorithms,
we also include plots for the two sided~\ID and the \SVD~of given rank (providing the
optimal approximation). Based on the plots, we make
three conjectures for matrices conditioned similar to those used in this example 
(note that \CUR-1 performs poorly in some of our other experiments):
\begin{itemize}
\item 
The accuracies of \CURID, \CUR-1, and \CUR-2, are all very similar. \CUR-H offers 
slightly worse approximations. 
\item 
The accuracy of \CUR-3 is worse than all other algorithms tested.
\item 
The two-sided \ID~is in every case more accurate than the \CUR-factorizations.
\end{itemize}

Next, in Figure \ref{fig:setIII}, we compare
the performance and runtimes of \CUR-H, \CUR-1, and \CUR-2 algorithms 
with the randomized \SVD~\cite{MR2806637} 
(which gives close results to the true \SVD~of given 
rank but at substantially less cost) and the \CURID~algorithm using the randomized \ID,
as described in this text (using $q=2$ in the power sampling scheme \eqref{eq:Ypower}).
This comparison allows us to test algorithms which can be used in practice on 
large matrices, since they involve randomization. We again use random matrices 
constructed as above whose singular values are logspaced, ranging from 
$10^{0}$ to $10^{-3}$, but of larger size: $2000\times4000$. 
We notice that the performance with all schemes is similar but the runtime with 
the randomized \CURID~algorithm is substantially lower than with the other schemes. 
The runtime of \CUR-1 is substantially greater than of the other 
schemes. The plotted quantities are again medians over $5$ trials.

In Figure \ref{fig:setIV}, we repeat the experiment using the randomized \SVD~with the
two matrices $\mtx{A}_1$ and $\mtx{A}_2$ defined in the preprint \cite{2014arXiv1407.5516S}.
The matrices $\mtx{A}_1,\mtx{A}_2 \in \real^{300,000\times300}$ are constructed as follows:
\begin{equation*}
\mtx{A}_1 = \displaystyle\sum_{j=1}^{10} \frac{2}{j} x_j y^T_j +
\displaystyle\sum_{j=11}^{300} \frac{1}{j} x_j y^T_j \quad \mbox{and} \quad
\mtx{A}_2 = \displaystyle\sum_{j=1}^{10} \frac{1000}{j} x_j y^T_j +
\displaystyle\sum_{j=11}^{300} \frac{1}{j} x_j y^T_j ,
\end{equation*}
where $x$ and $y$ are sparse vectors with random non-negative entries.
One problem with using traditional \CUR~algorithms for these matrices stems
from the fact that the singular values of $\mtx{A}_1$ and $\mtx{A}_2$ decay rapidly. 
Due to this, the performance of \CUR-1 (and of \CUR-3, which we do not show) 
for these examples is poor. It appears that this is because for these schemes, 
the rapid decay of the singular values of the input matrix translates into 
the inversion of ill-conditioned matrices, which adversely effects performance.
On the other hand, \CURID~and \CUR-2 offer similar performance, close to the approximate 
\SVD~results. In Figure \ref{fig:setIV}, we show the medians of relative errors versus 
$k$ over $5$ trials. 

In Figure \ref{fig:set_cur_and_other_algs}, we show comparison between absolute errors 
given by our non-randomized 
and randomized \CURID~algorithms and the truncated \SVD~and \QR~factorizations in 
terms of the square of the Frobenius norm and the spectral norm. 
We use $600\times600$ test matrices, with varying singular value decay, as before.
In particular, we check here if the optimistic bound:
\begin{equation}
\|\mtx{A} - \mtx{C}\mtx{U}\mtx{R} \|_F^2 \leq (1 + \epsilon) \|\mtx{A} - \mtx{A}_k\|_F^2 \quad \mbox{with} \quad \mtx{A}_k = \mtx{U}_k \mtx{\Sigma}_k \mtx{V}_k^*
\end{equation}
from \cite{boutsidis2014optimal} holds with $1 < \epsilon < 2$ for the 
non-randomized \CURID~scheme. For 
$\epsilon \approx 2$ and $k \ll \min(m,n)$ the bound sometimes holds, 
but it does not hold for all $k$. Despite this, we may also 
observe from the bottom row of Figure \ref{fig:set_cur_and_other_algs} that for matrices 
with rapid singular value decay, the \CURID~error in the spectal norm is sometimes lower 
even than that of the truncated \QR.  

In Figure \ref{fig:setV}, we have an image compression experiment,
using \CURID~and \CUR-1,\CUR-2, and \CUR-H with the full \SVD. 
We take two black and white images (of size $350\times507$ and $350\times526$)
and transform the matrix using four levels of the $2 D$ CDF 97 wavelet transform. 
We then threshold the result, leaving a sparse $m\times n$ matrix $\mtx{M}$ with about 
$30\%$ nonzeros (with same dimensions as the original image). 
Then we go on to construct a low rank \CUR~approximation
of this wavelet thresholded matrix (with $k=\min(m,n)/15$)
to further compress the image data. Storing the three
matrices $\mtx{C}$, $\mtx{U}$, and $\mtx{R}$ corresponds to storing about $8$ time less 
nonzeros vs storing $\mtx{M}$. 
To reconstruct the image from this compressed form, we perform the inverse 
CDF 97 WT transform on the matrix product $\mtx{C U R}$, which approximates the wavelet 
thresholded matrix. From the plots, we see that \CURID~produces a $\mtx{U}$ which 
has less rapid singular value decay than the $\mtx{U}$ matrix obtained with the  
\CUR-1 and \CUR-H algorithms. In particular, the reconstructions obtained with 
\CUR-1 are very poor and the \mtx{U} obtained from this scheme has rapidly 
decaying singular values, comparable to those of $\mtx{M}$.

Thus, in each case, we observe comparable or even better performance
with \CURID~than with existing \CUR~algorithms. For large matrices, existing
\CUR~algorithms that rely on the singular vectors must be used in conjunction
with an accelerated scheme for computing approximate singular vectors, such as,
e.g., the randomized method of \cite{MR2806637}, or to use \CURID~with the 
randomized \ID. We find that for random matrices the performance is similar, but 
\CURID~is easier to implement and is generally more efficient.
Also, as in the case of the imaging example we present,
existing \CUR~algorithms suffer from a badly conditioned
$\mtx{U}$ matrix when the original matrix is not well conditioned. The $\mtx{U}$ matrix returned
by the \CURID~algorithm tends to be better conditioned.

Finally, we again remark that optimized codes for the algorithms we propose are available as 
part of the RSVDPACK software package \cite{voronin2015rsvdpack}.

\newpage

\begin{figure*}[h!]
\centerline{
\includegraphics[scale=0.38]{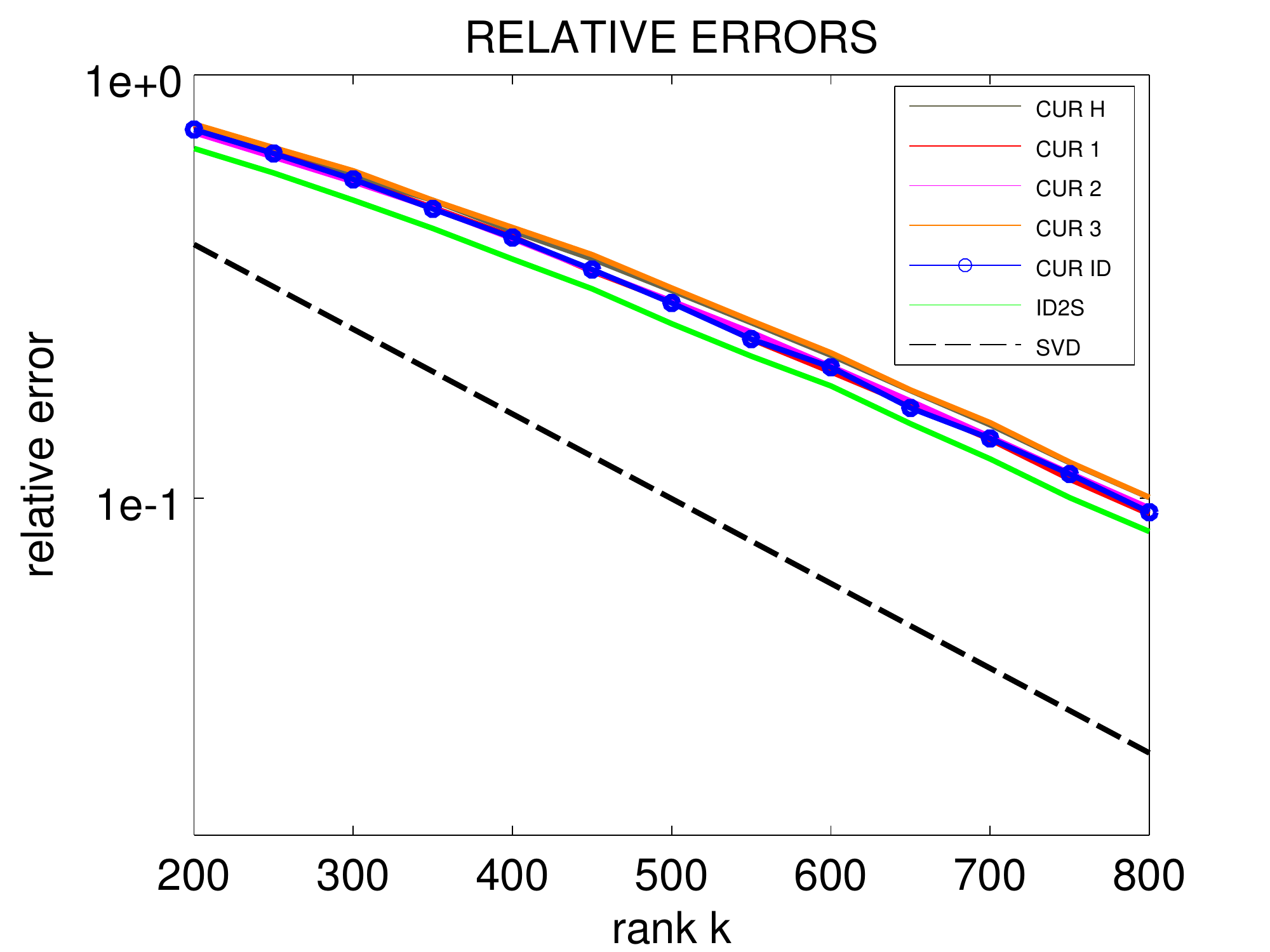}
\includegraphics[scale=0.38]{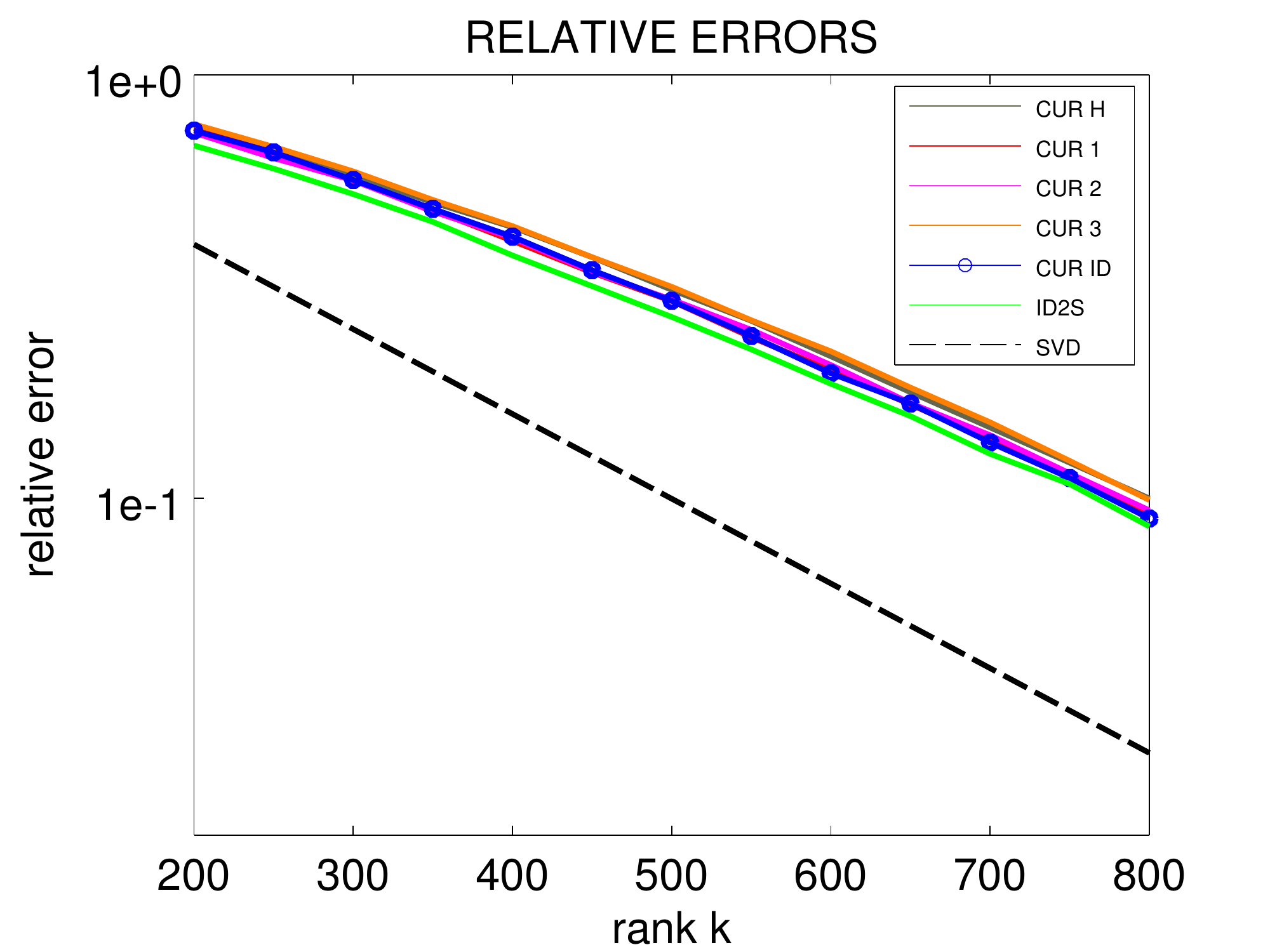}
}
\vspace{1.mm}
\centerline{
\includegraphics[scale=0.38]{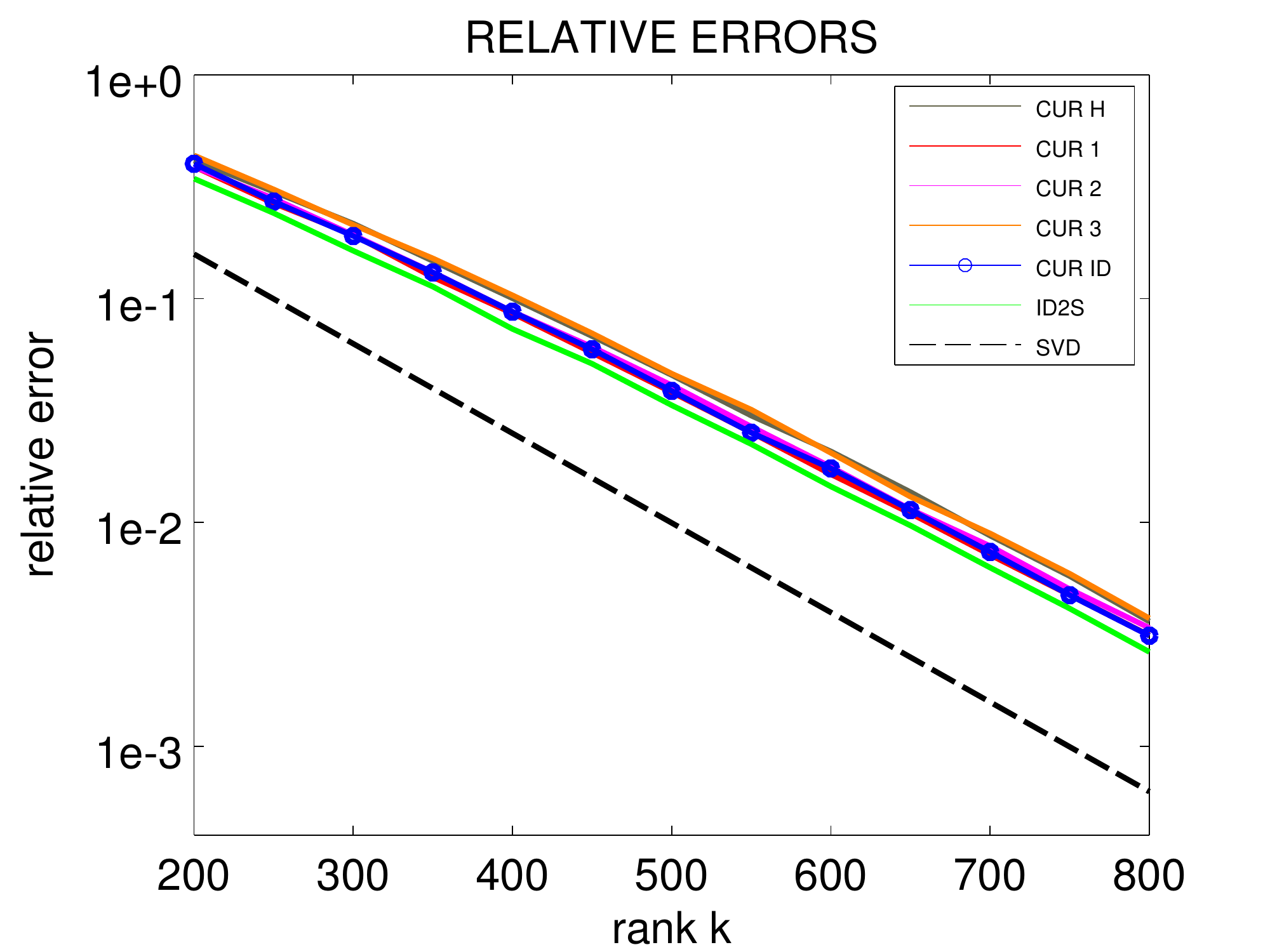}
\includegraphics[scale=0.38]{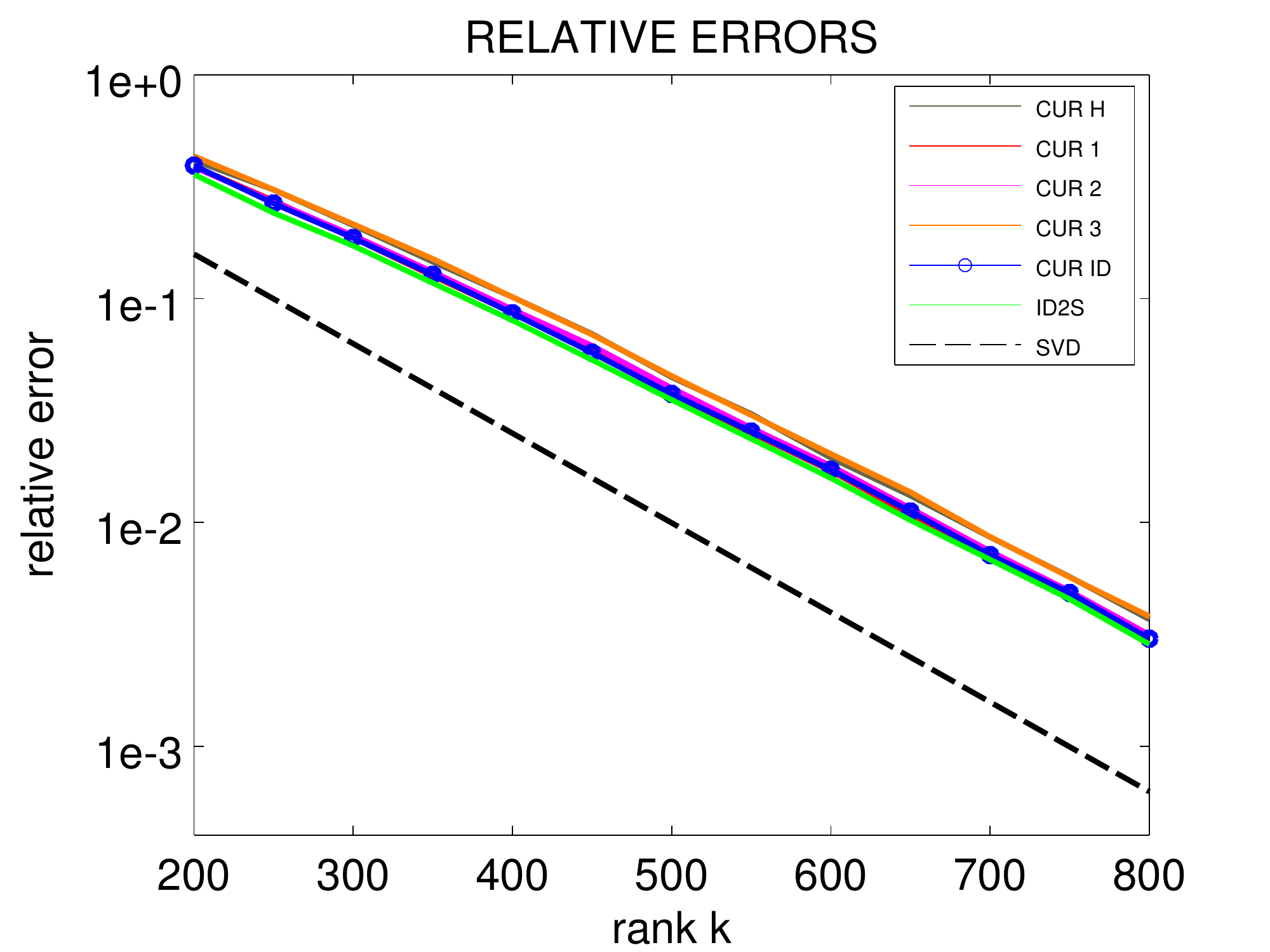}
}
\vspace{1.mm}
\centerline{
\includegraphics[scale=0.38]{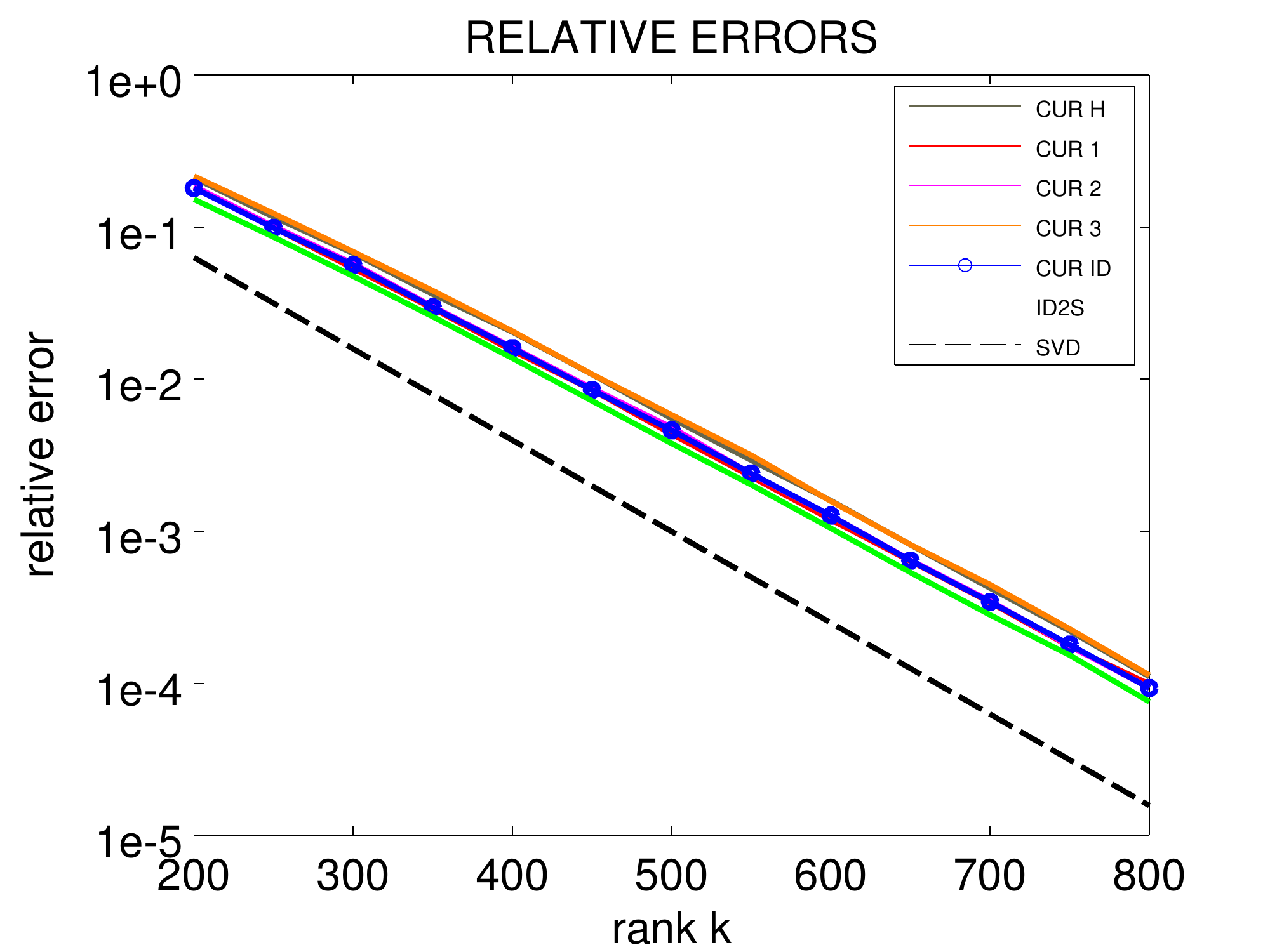}
\includegraphics[scale=0.38]{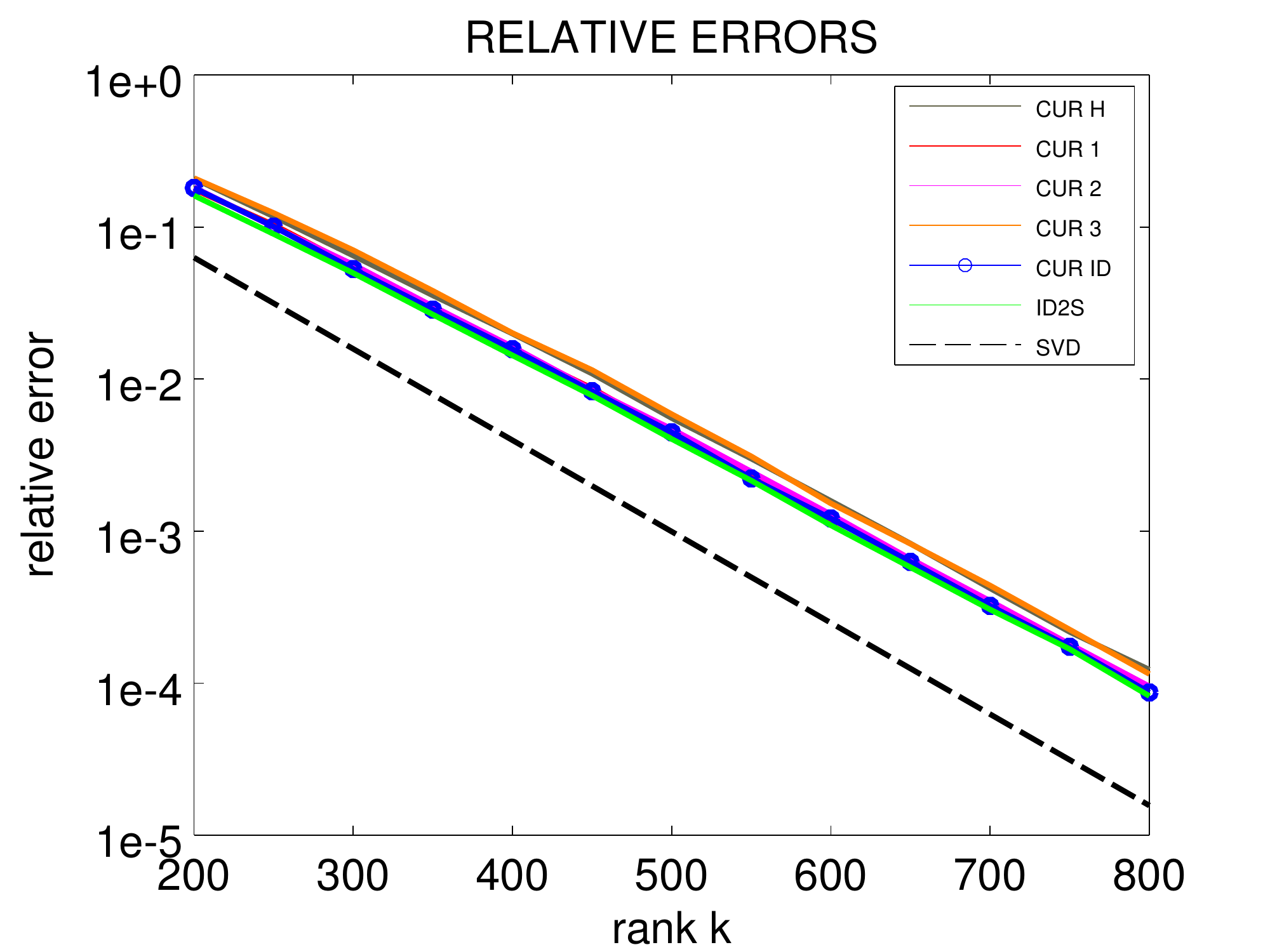}
}
\caption{Relative errors for differently conditioned matrices approximated with various algorithms. Left: fat matrices ($1000\times3000$), right: thin
matrices ($3000\times1000$). Top to bottom: faster drop off of logspaced singular values.}
\label{fig:setI_and_II}
\end{figure*}

\newpage

\begin{figure*}[h!]
\centerline{
\includegraphics[scale=0.30]{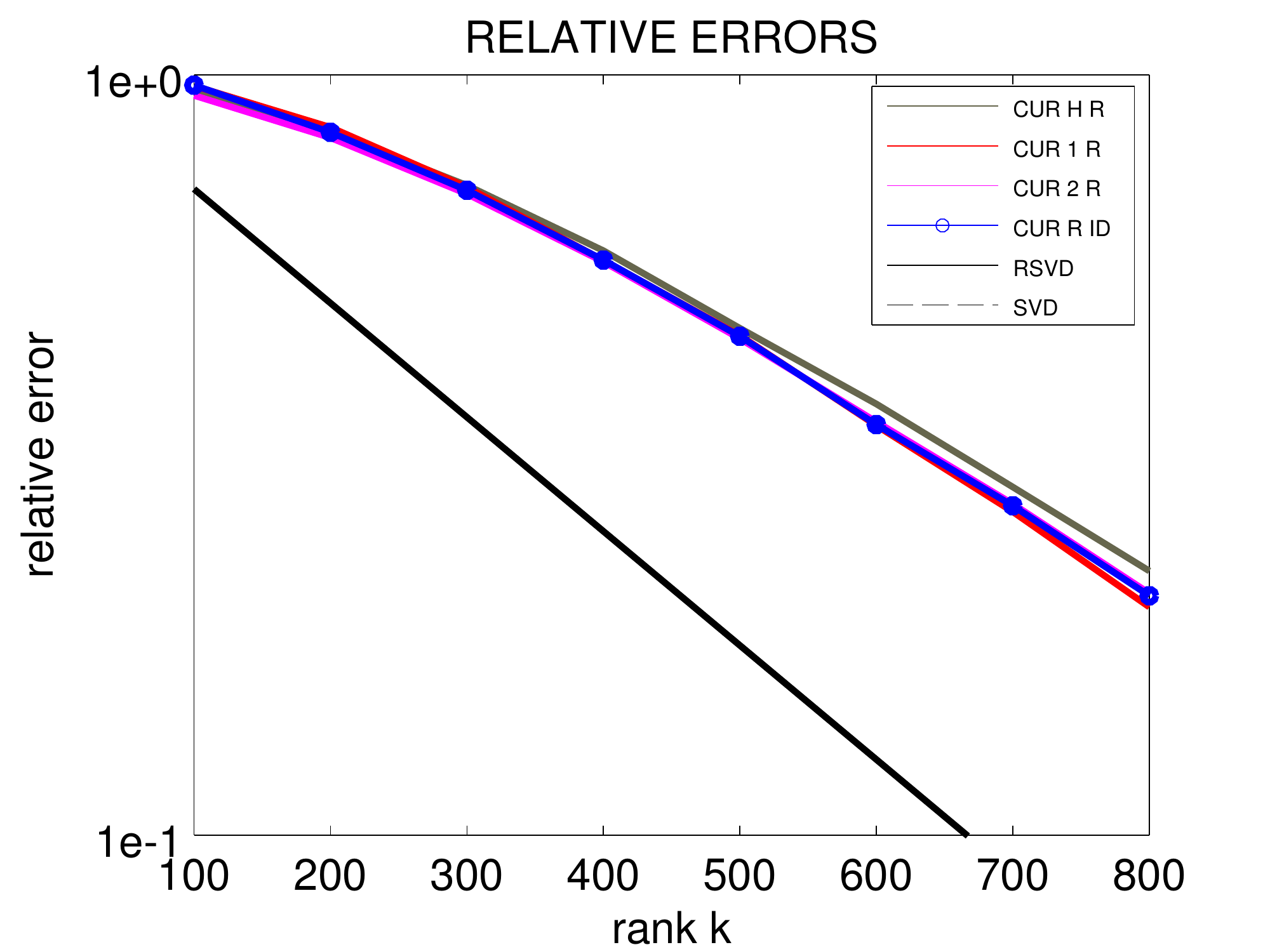}
\includegraphics[scale=0.30]{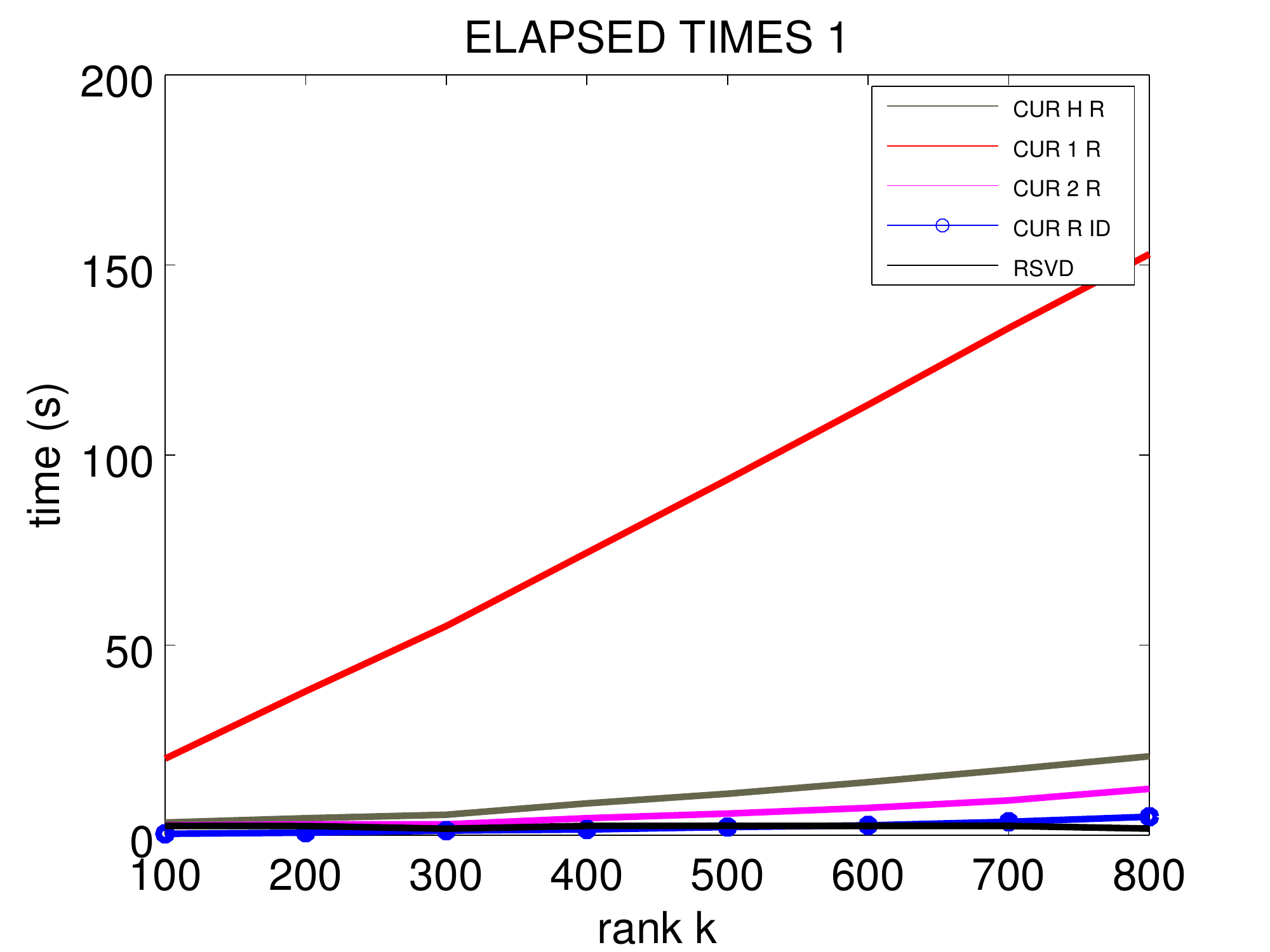}
\includegraphics[scale=0.30]{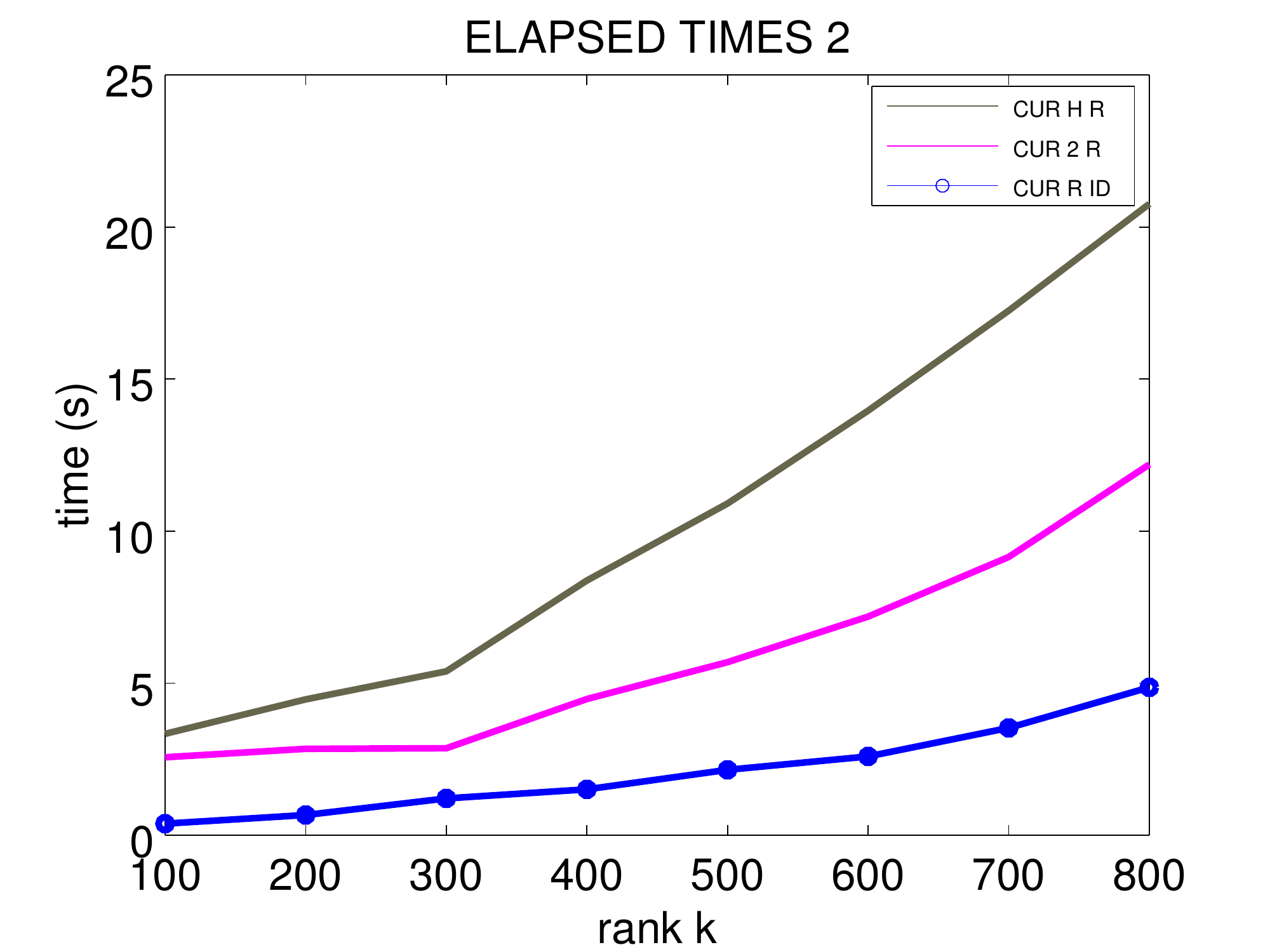}
}
\caption{Relative errors and elapsed times for \CUR-H,\CUR-1,\CUR-2 with randomized \SVD~and
\CURID~with the randomized \ID~using larger matrices of size $2000 \times 4000$. First time 
plot shows runtimes for all algorithms. Second time plot shows runtimes of \CUR-H, 
\CUR-2, and \CURID.}
\label{fig:setIII}
\end{figure*}

\begin{figure*}[h!]
\centerline{
\includegraphics[scale=0.45]{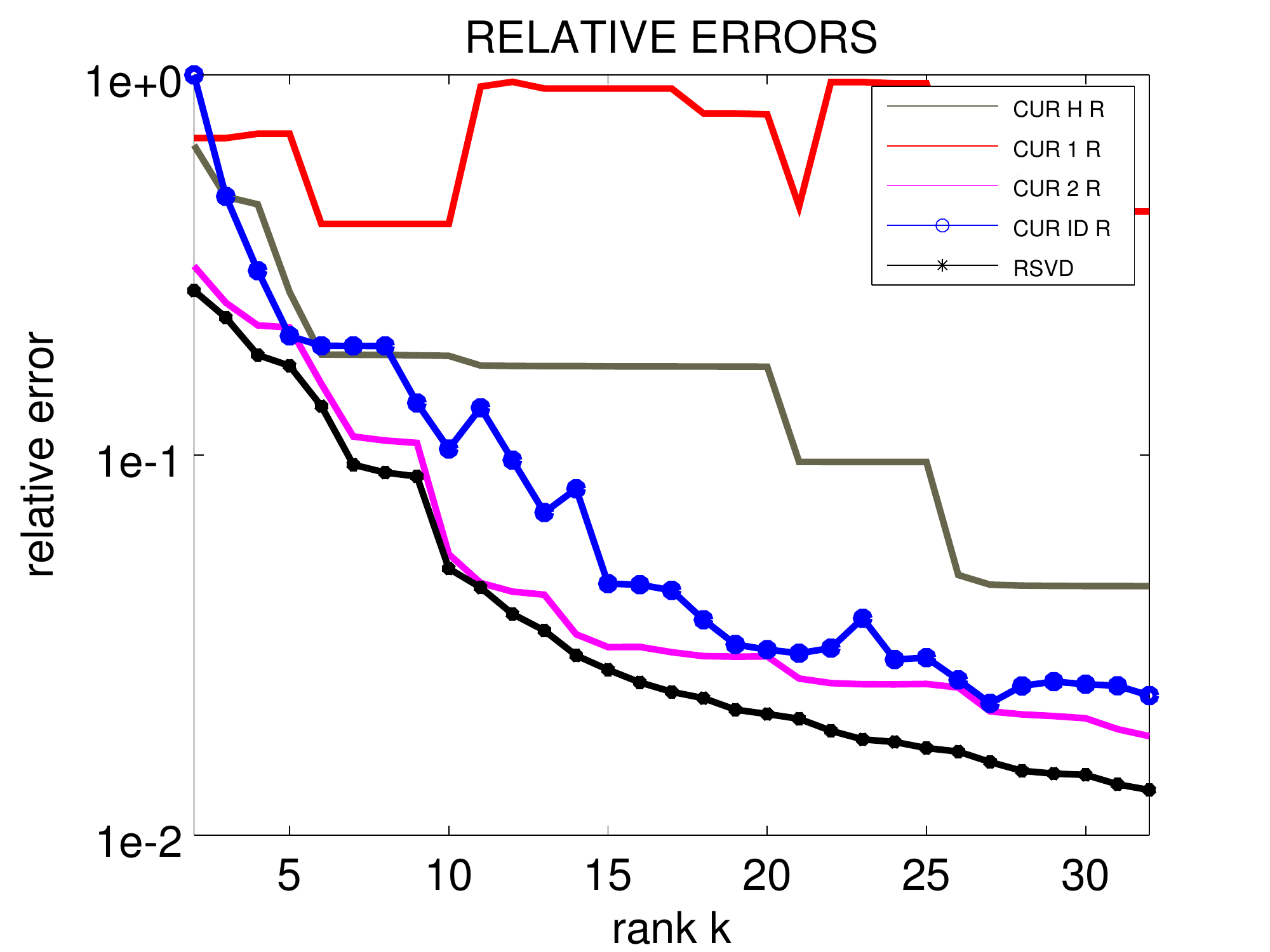}
\includegraphics[scale=0.45]{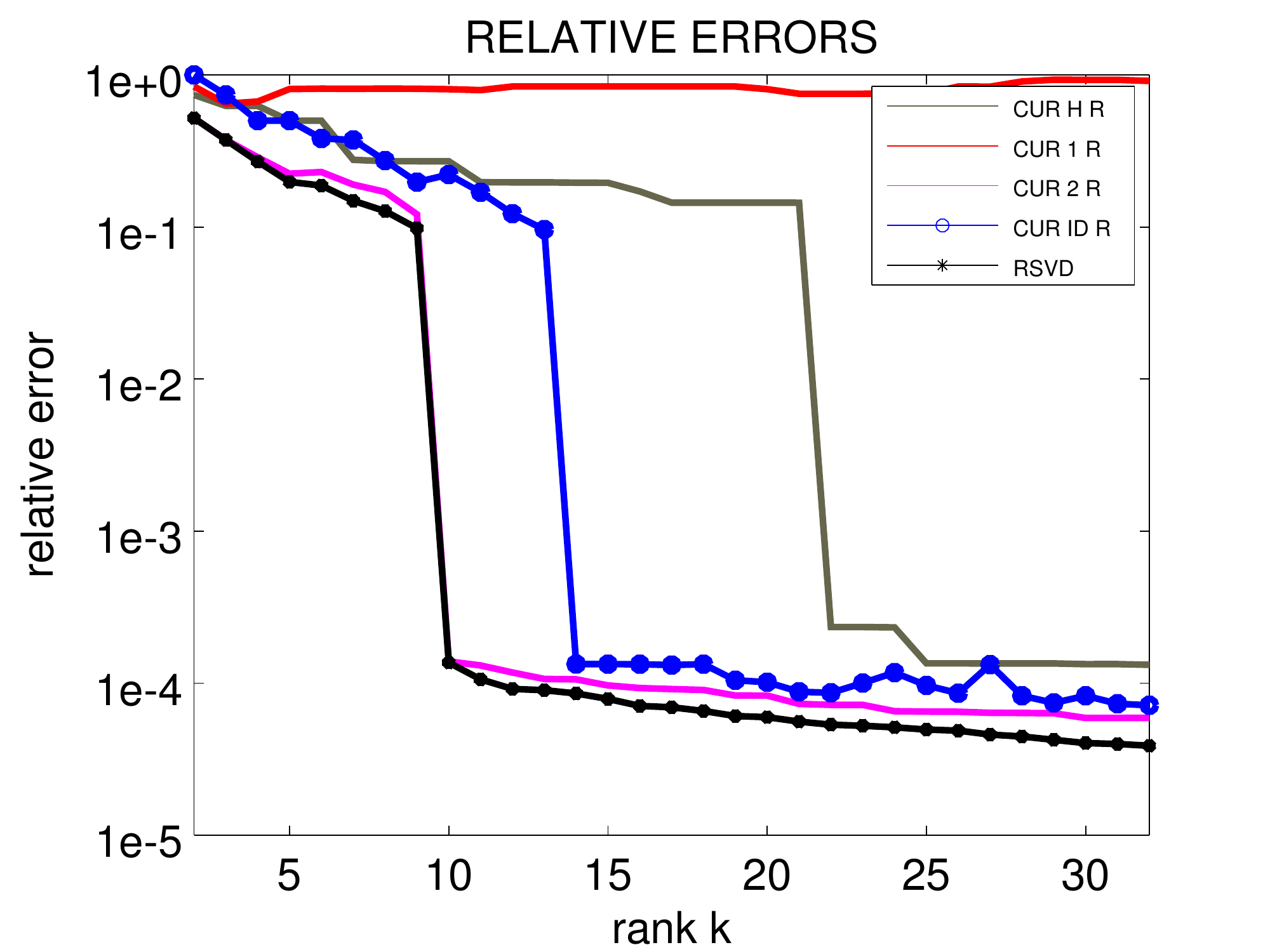}
}
\caption{Relative errors versus $k$ for matrices $\mtx{A}_1$ (left) and 
$\mtx{A}_2$ (right) from \cite{2014arXiv1407.5516S} approximated using 
\CUR-H,\CUR-1,\CUR-2 with randomized \SVD~and \CURID~with the randomized \ID.}
\label{fig:setIV}
\end{figure*}

\newpage

\begin{figure*}[h!]
\centerline{
\includegraphics[scale=0.22]{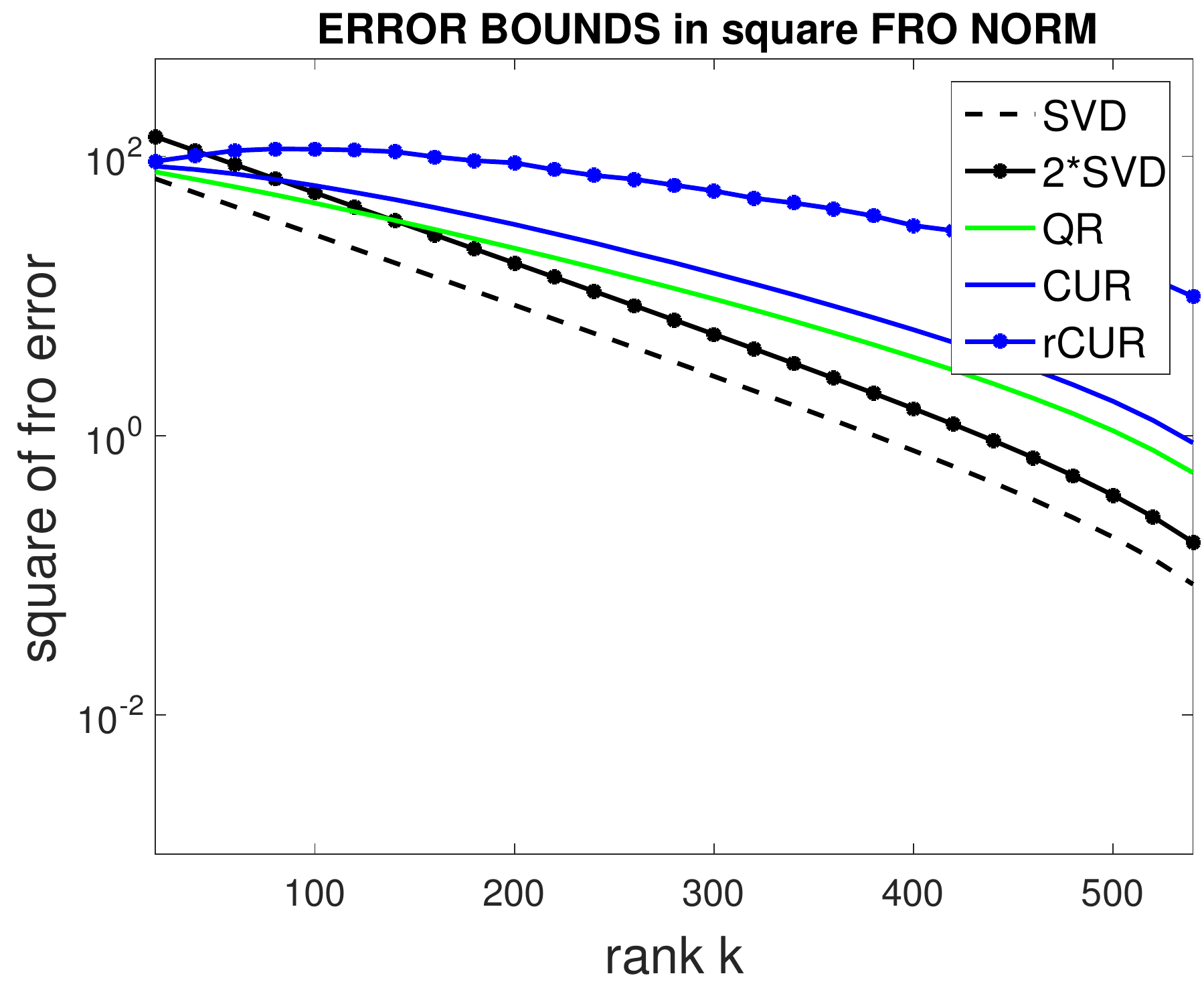}
\includegraphics[scale=0.22]{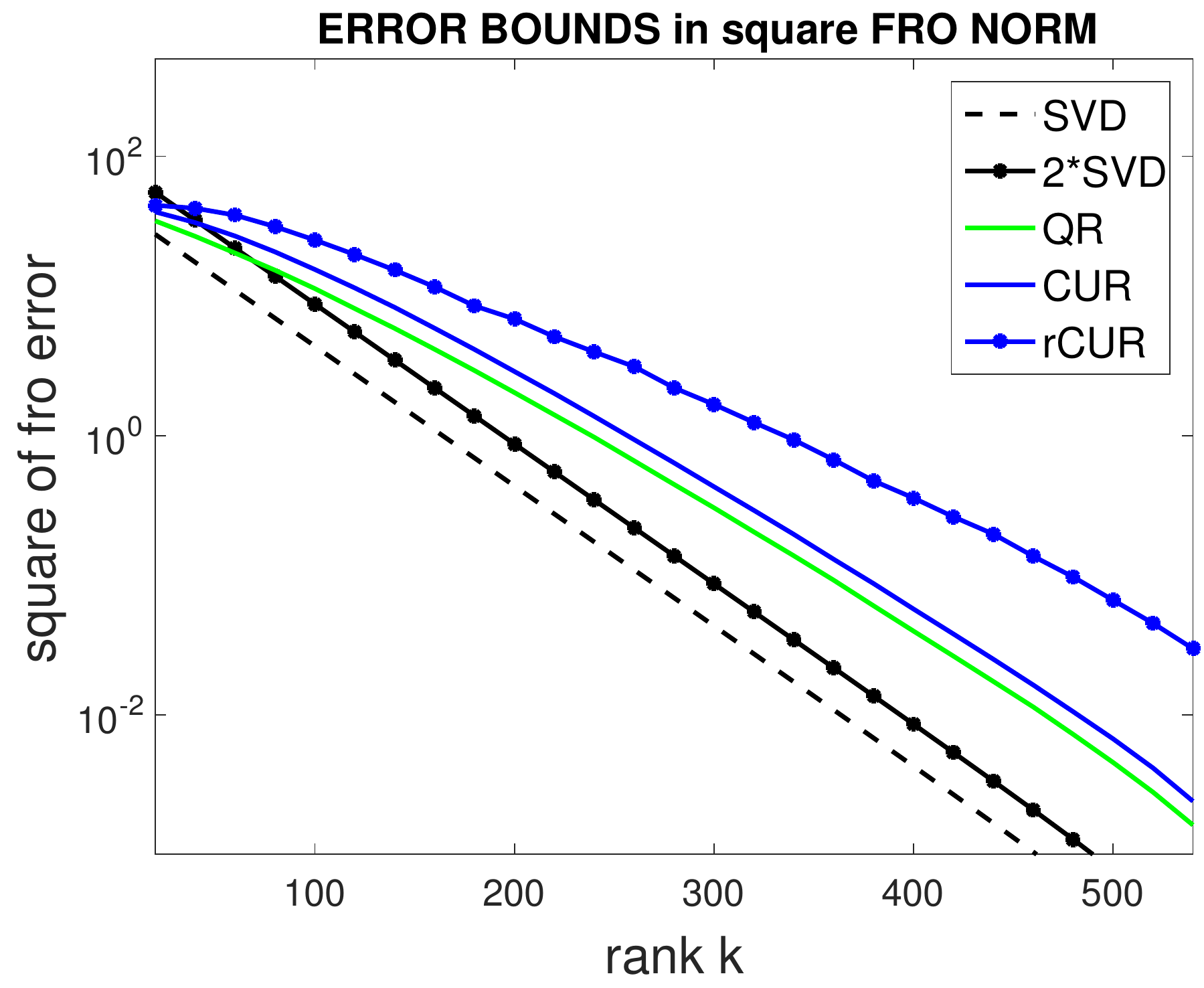}
\includegraphics[scale=0.22]{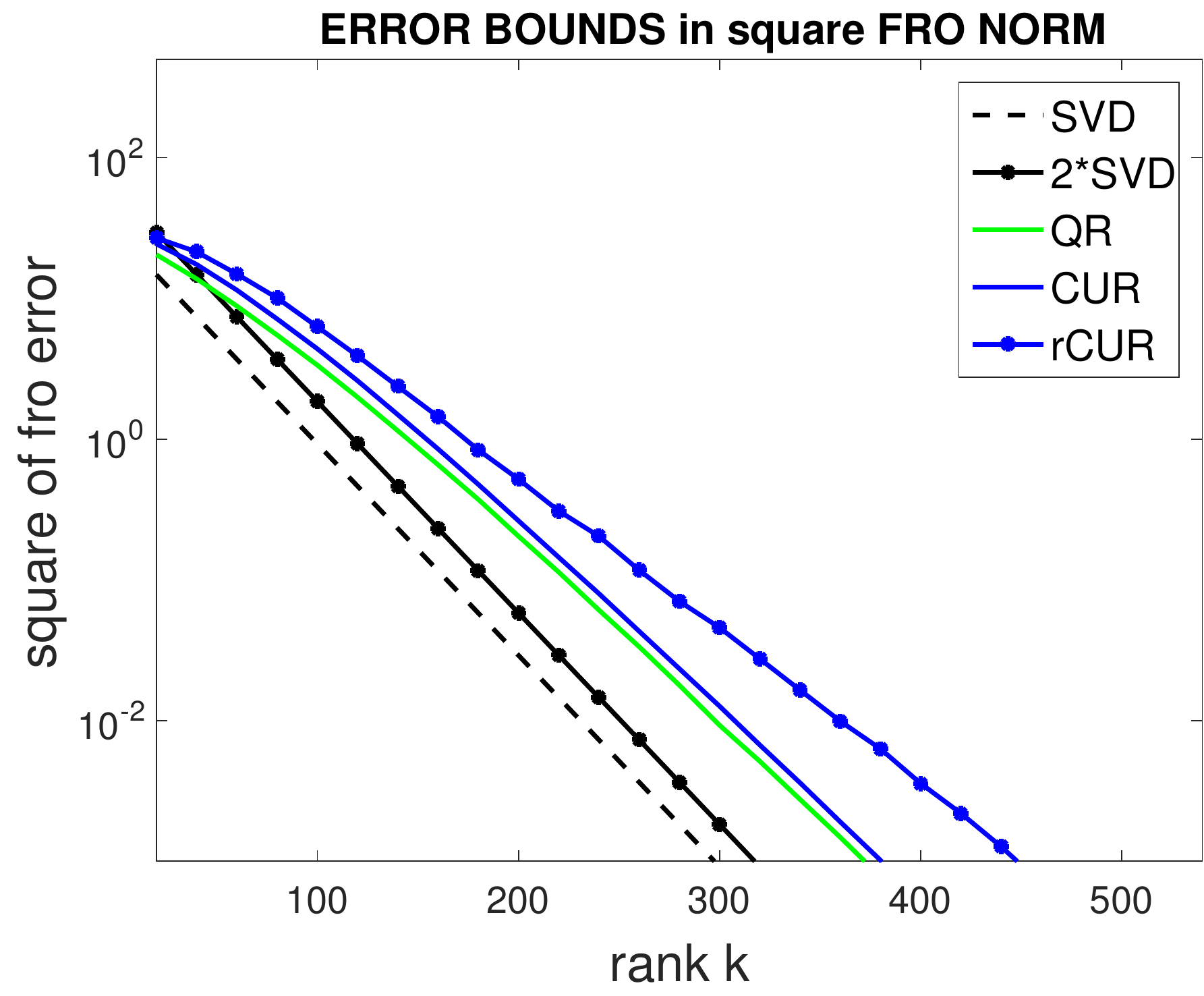}
}
\centerline{
\includegraphics[scale=0.22]{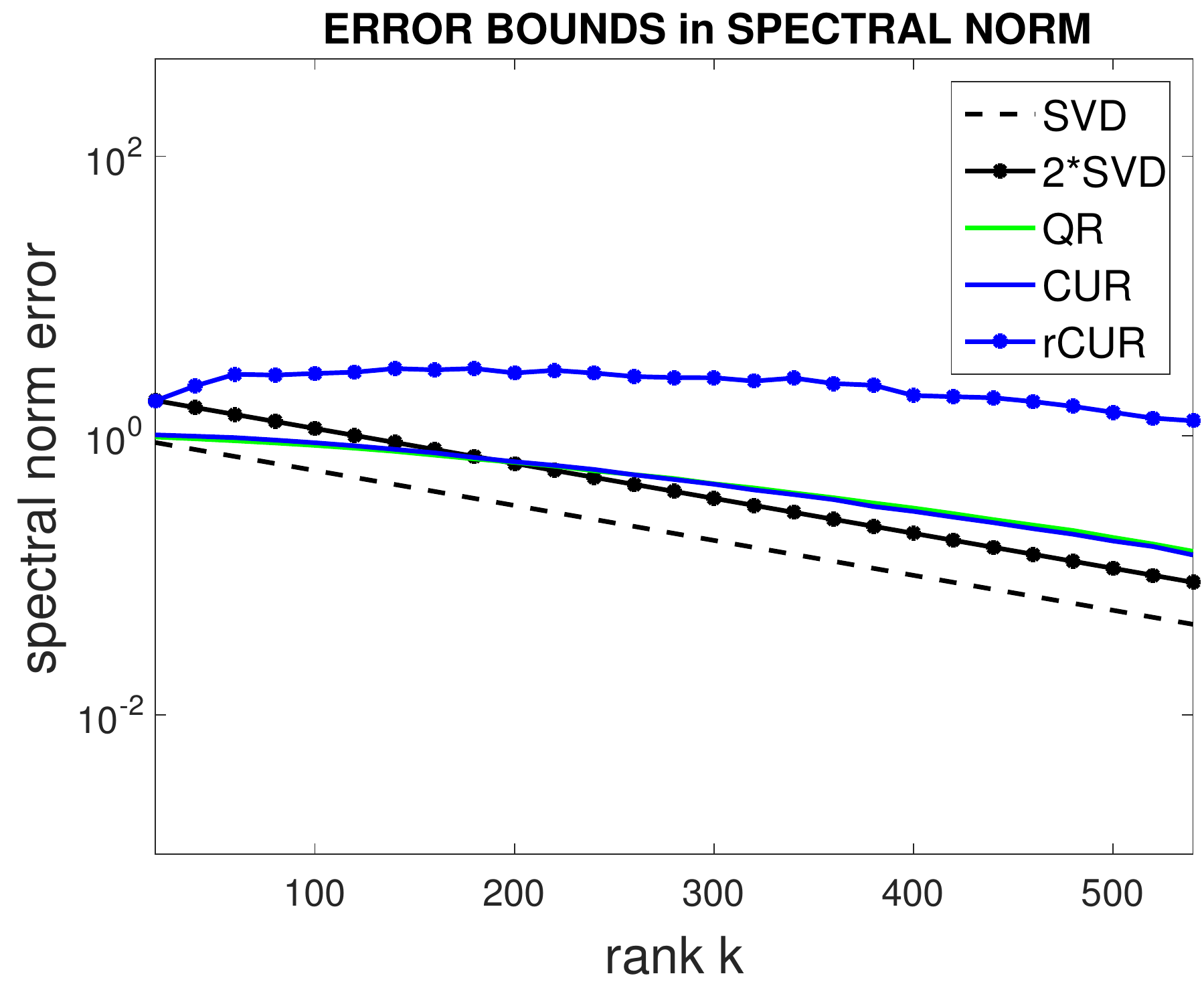}
\includegraphics[scale=0.22]{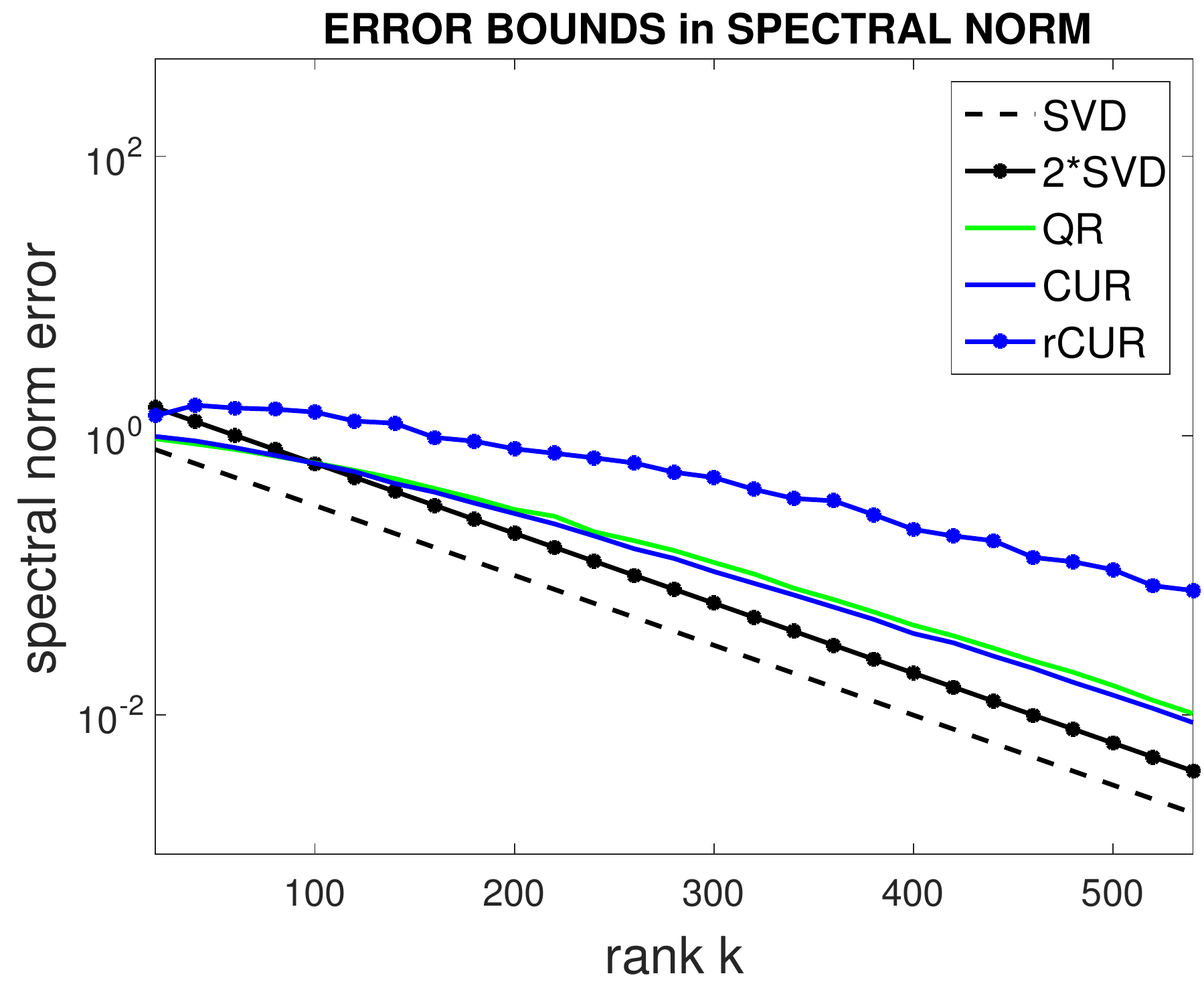}
\includegraphics[scale=0.22]{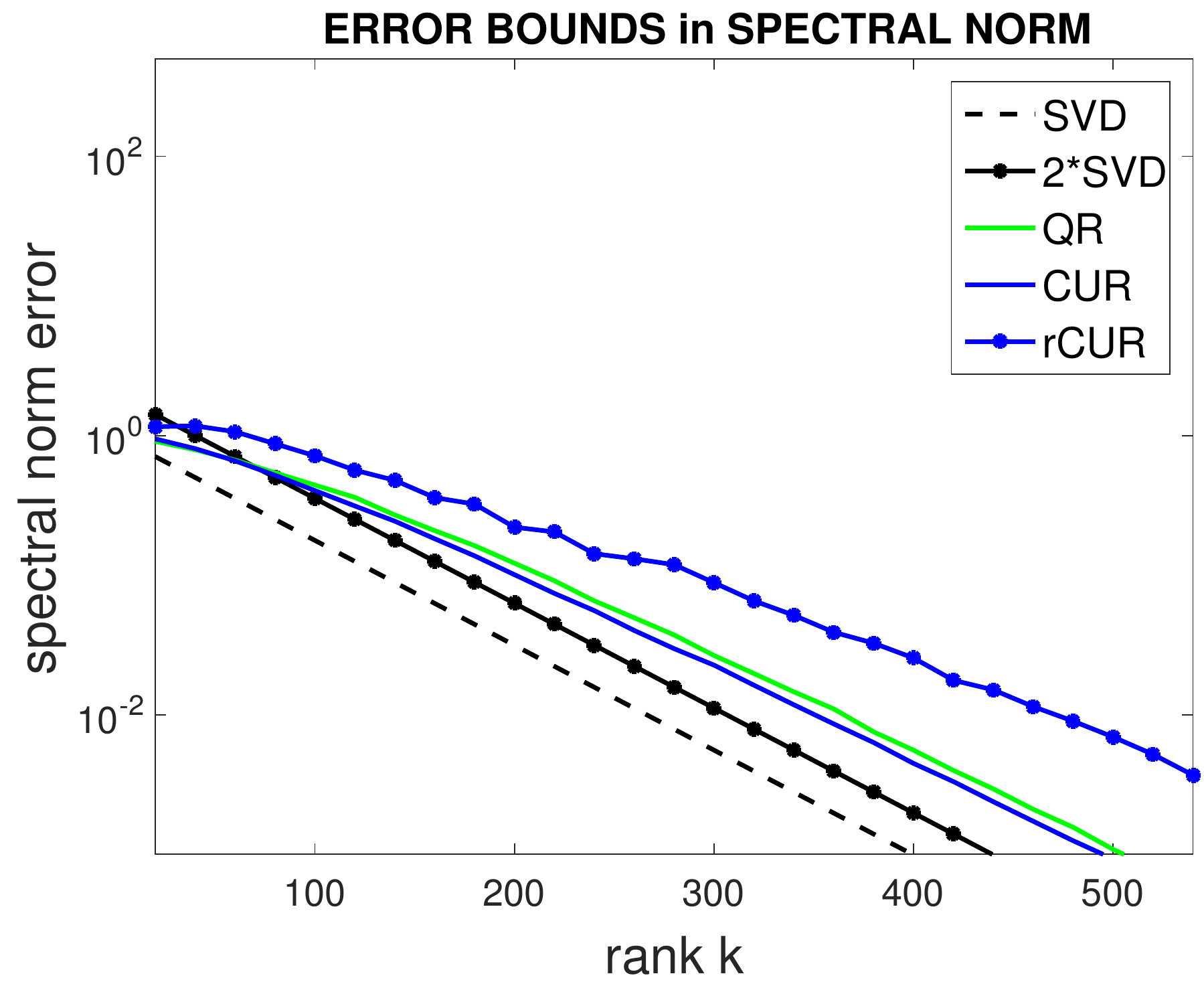}
}
\caption{Comparison of absolute error bounds for rank $k$ \CURID~and \CURID~with 
randomization in comparison to truncated rank $k$ \SVD~and truncated \QR~decompositions 
in terms of square Frobenius norm (top) and spectral norm (bottom) for matrices with 
singular values distributed on a logarithmic scale between 
$1$ and $10^{-b}$ with $b=1.5,3,4.5$.}
\label{fig:set_cur_and_other_algs}
\end{figure*}

\begin{figure*}[h!]
\centerline{
\includegraphics[scale=0.211]{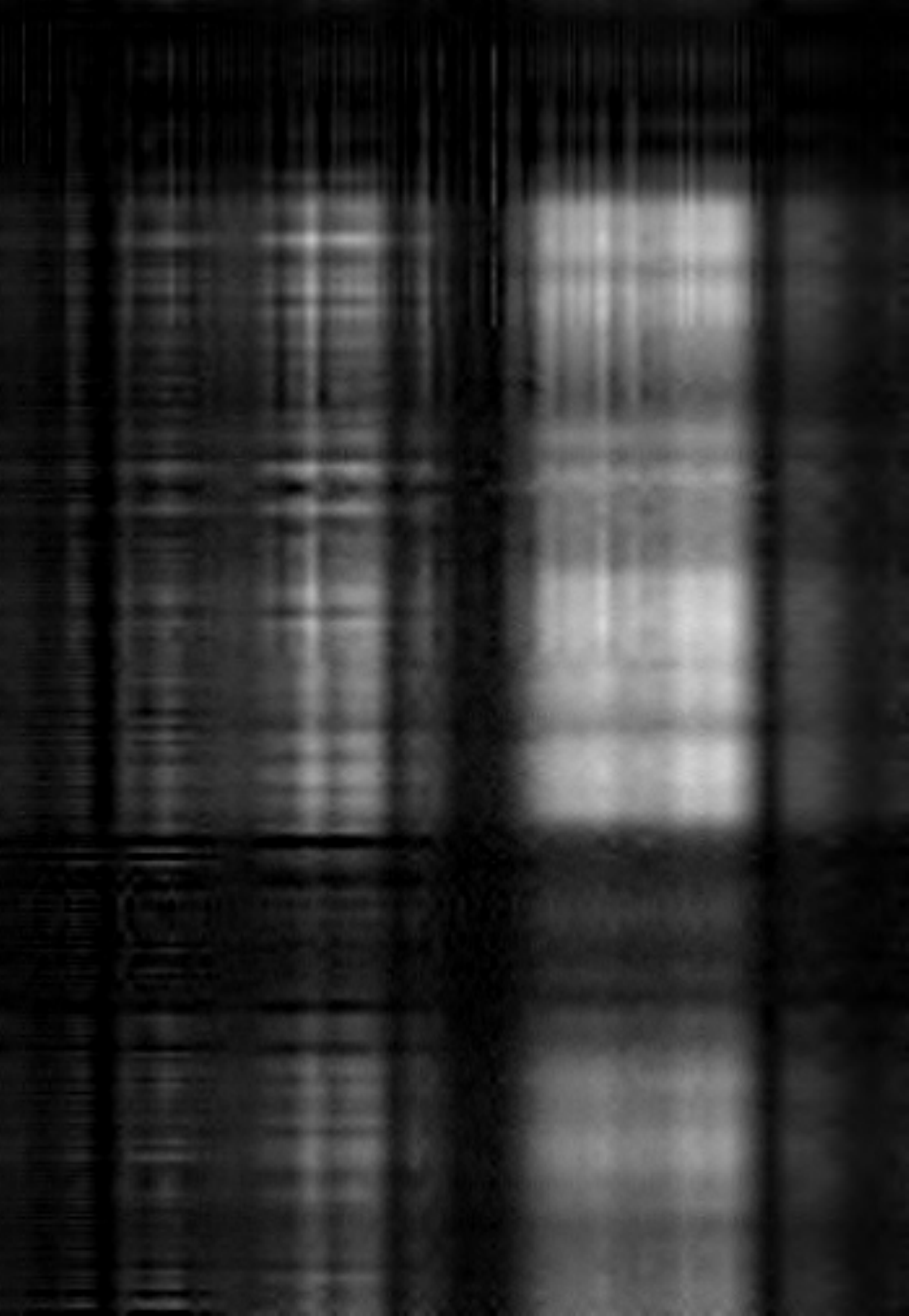}
\includegraphics[scale=0.211]{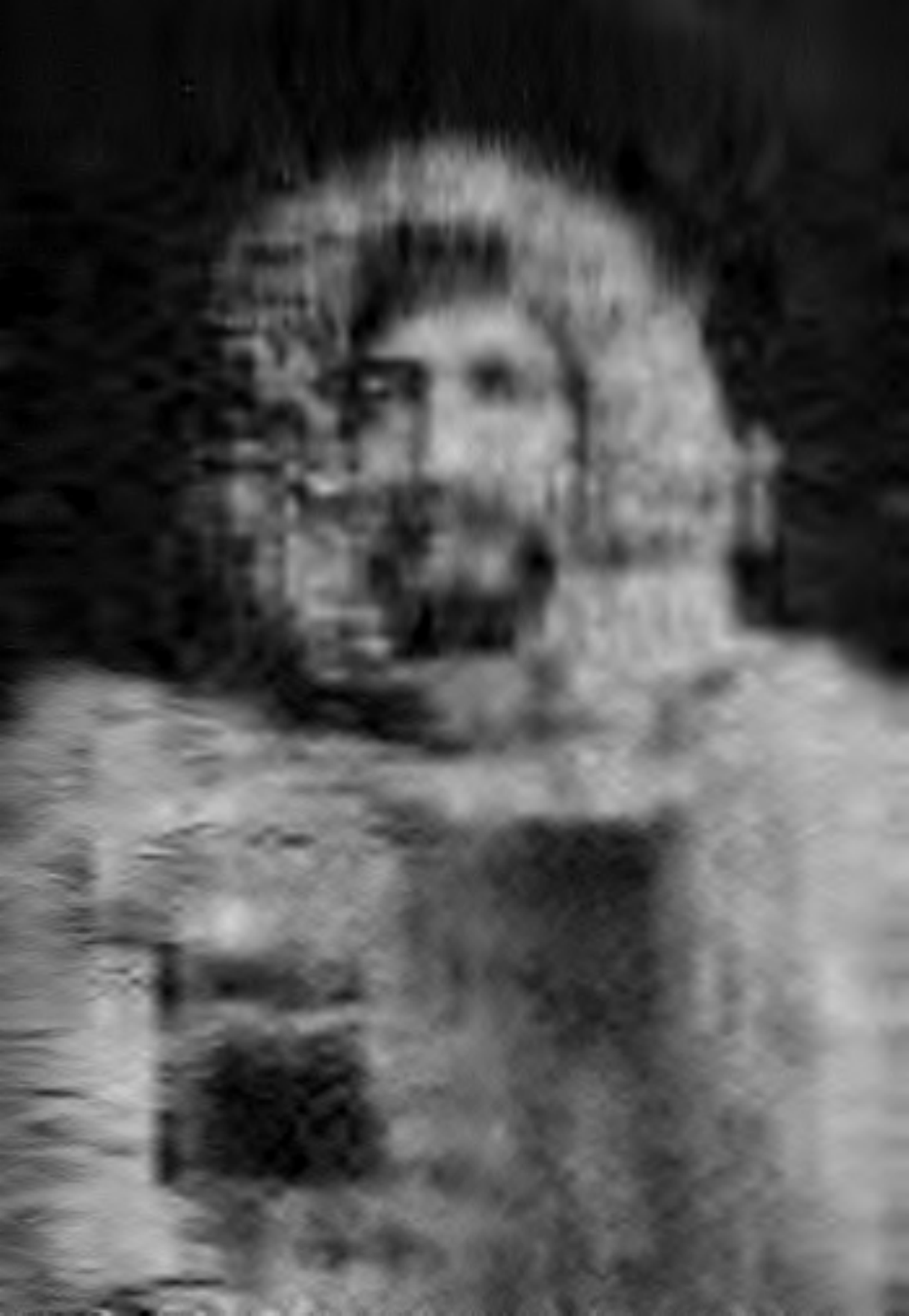}
\includegraphics[scale=0.211]{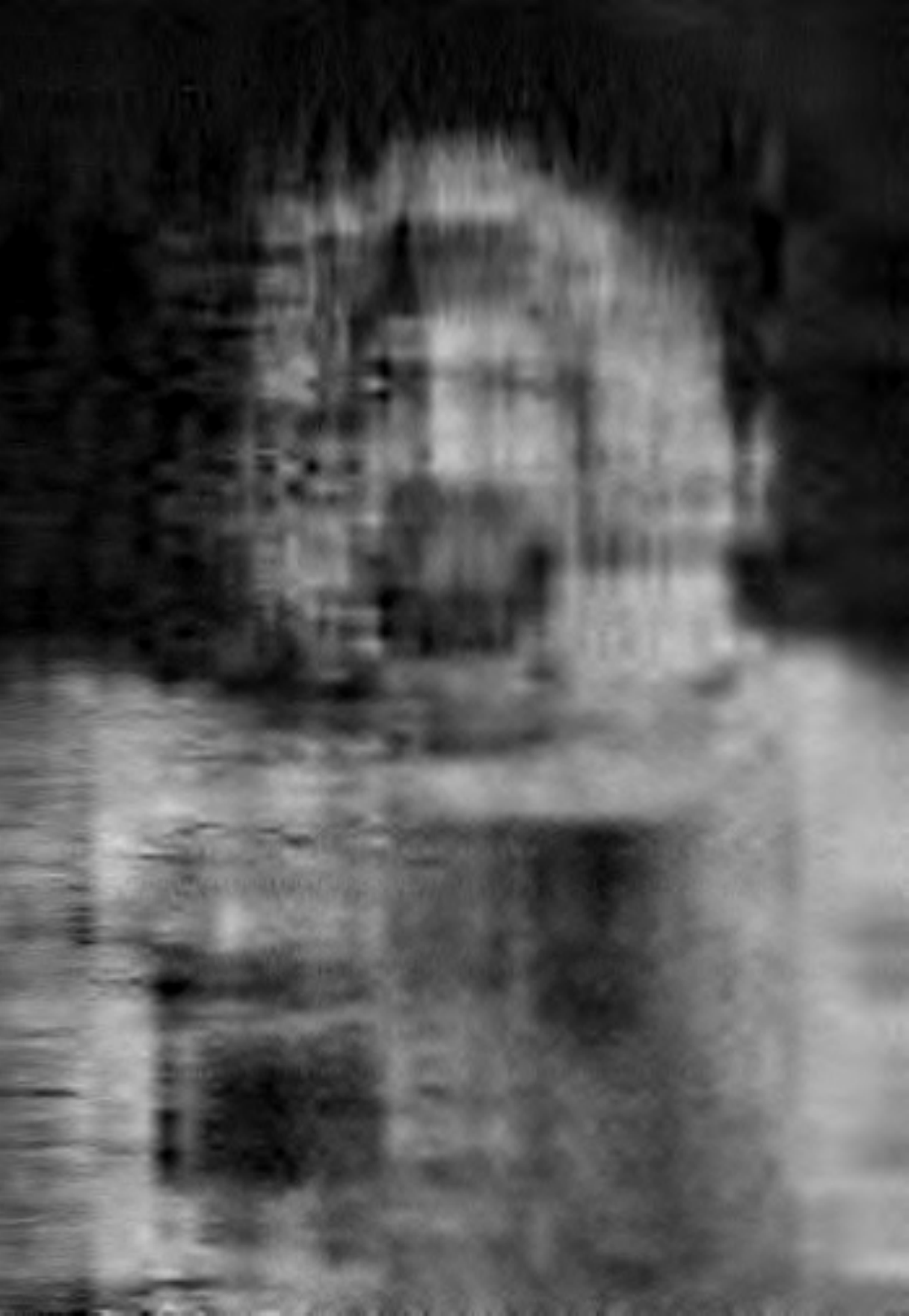}
\includegraphics[scale=0.211]{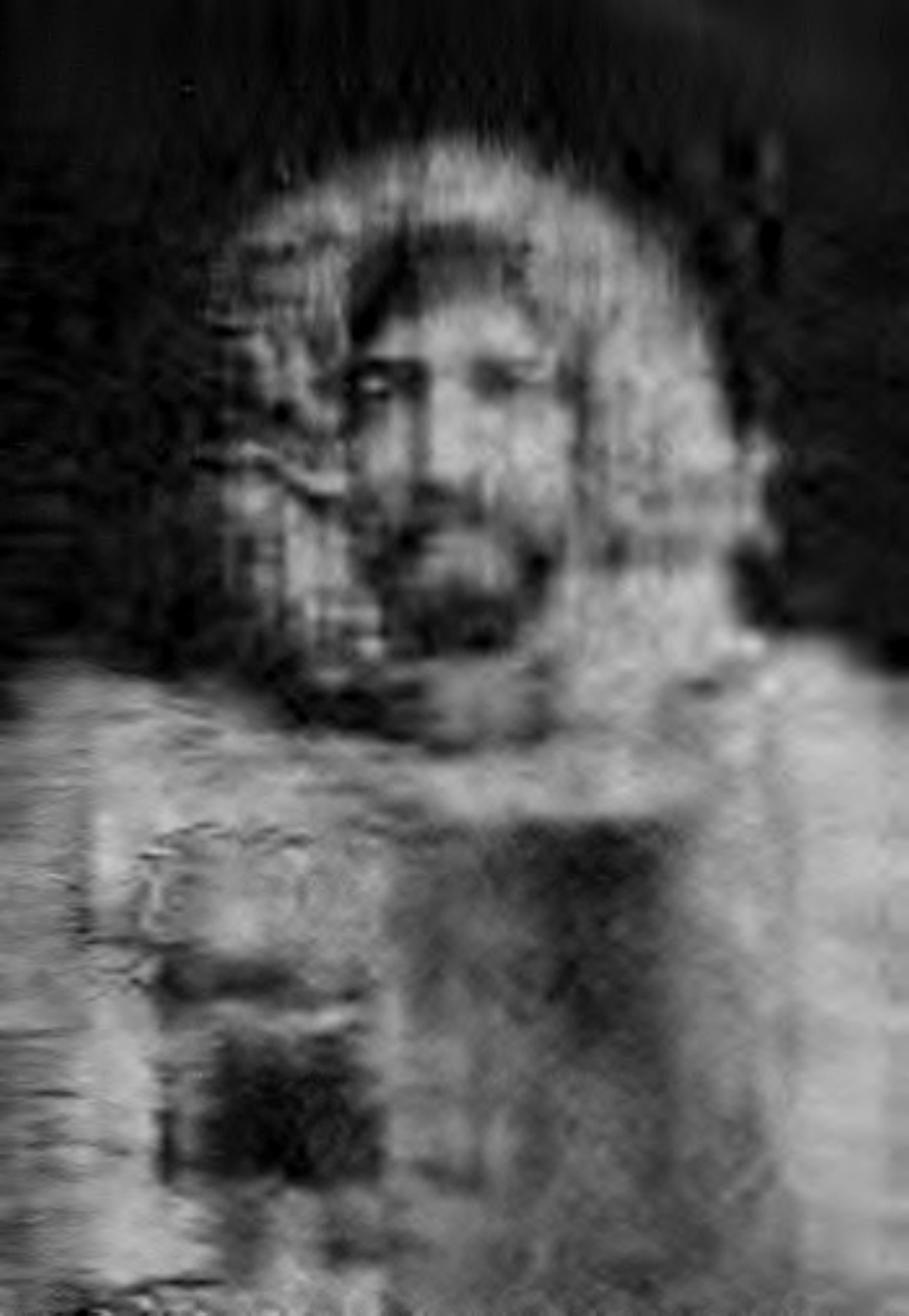}
\includegraphics[scale=0.33]{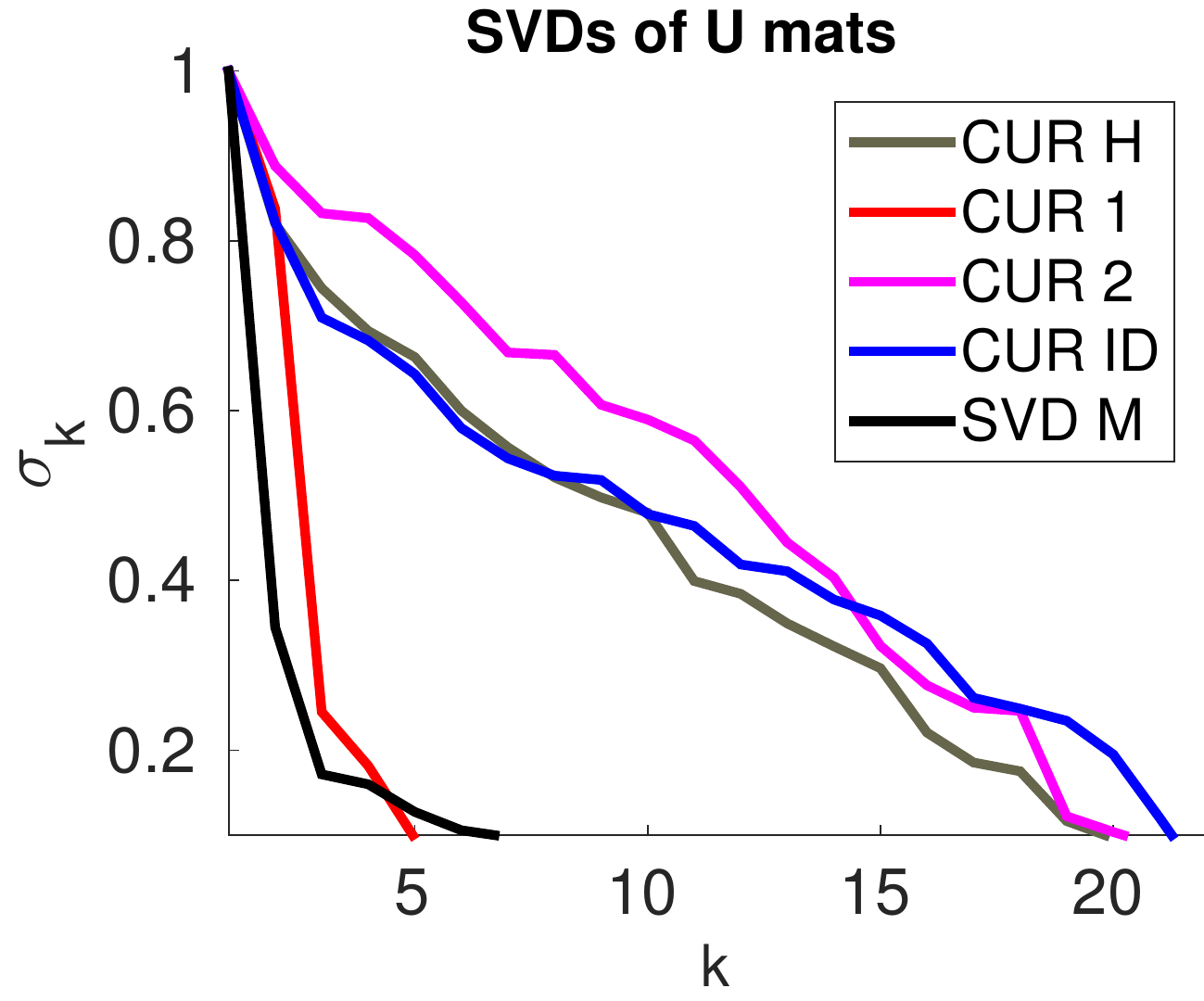}
}
\centerline{
\includegraphics[scale=0.2098]{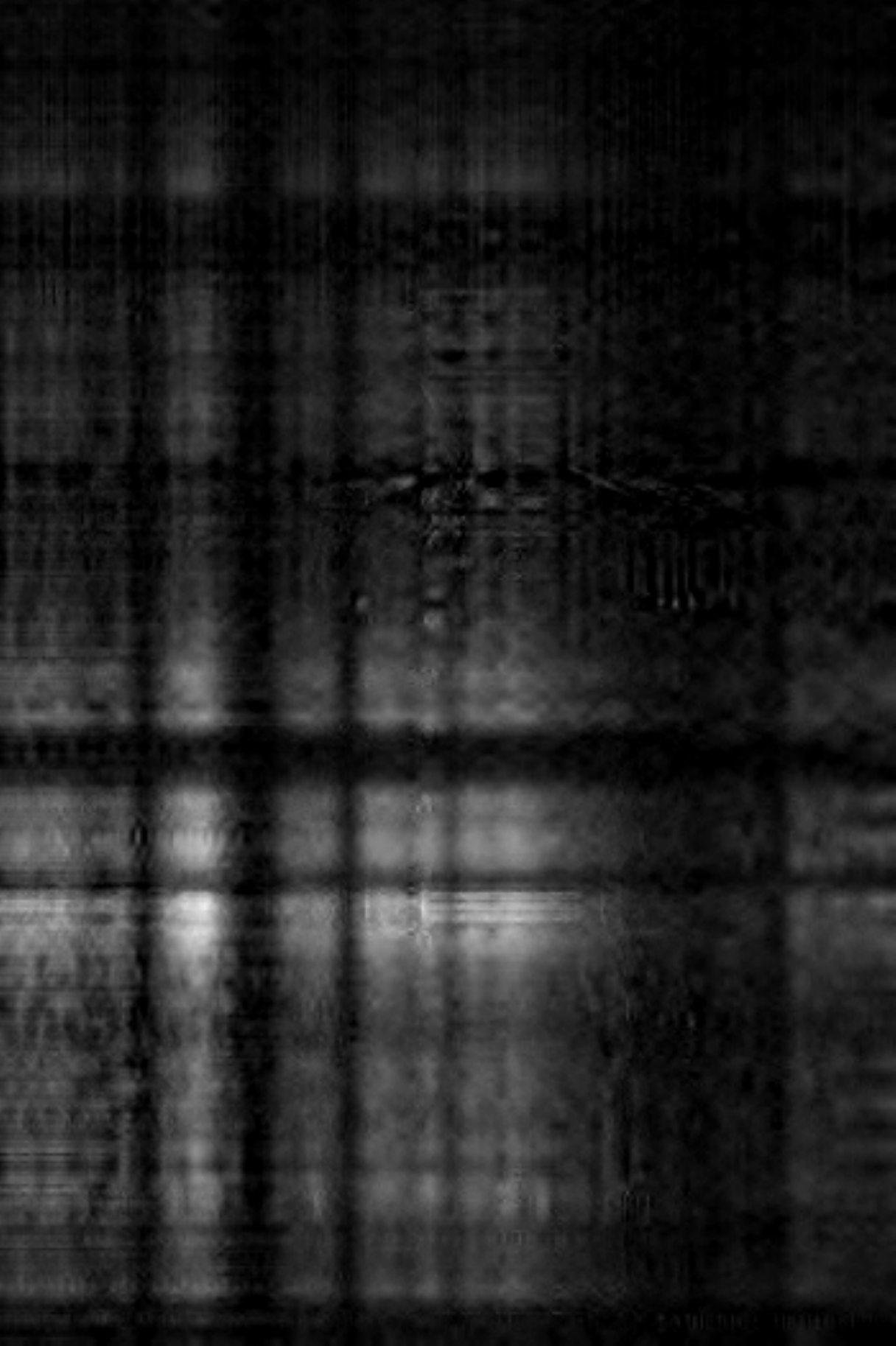}
\includegraphics[scale=0.2098]{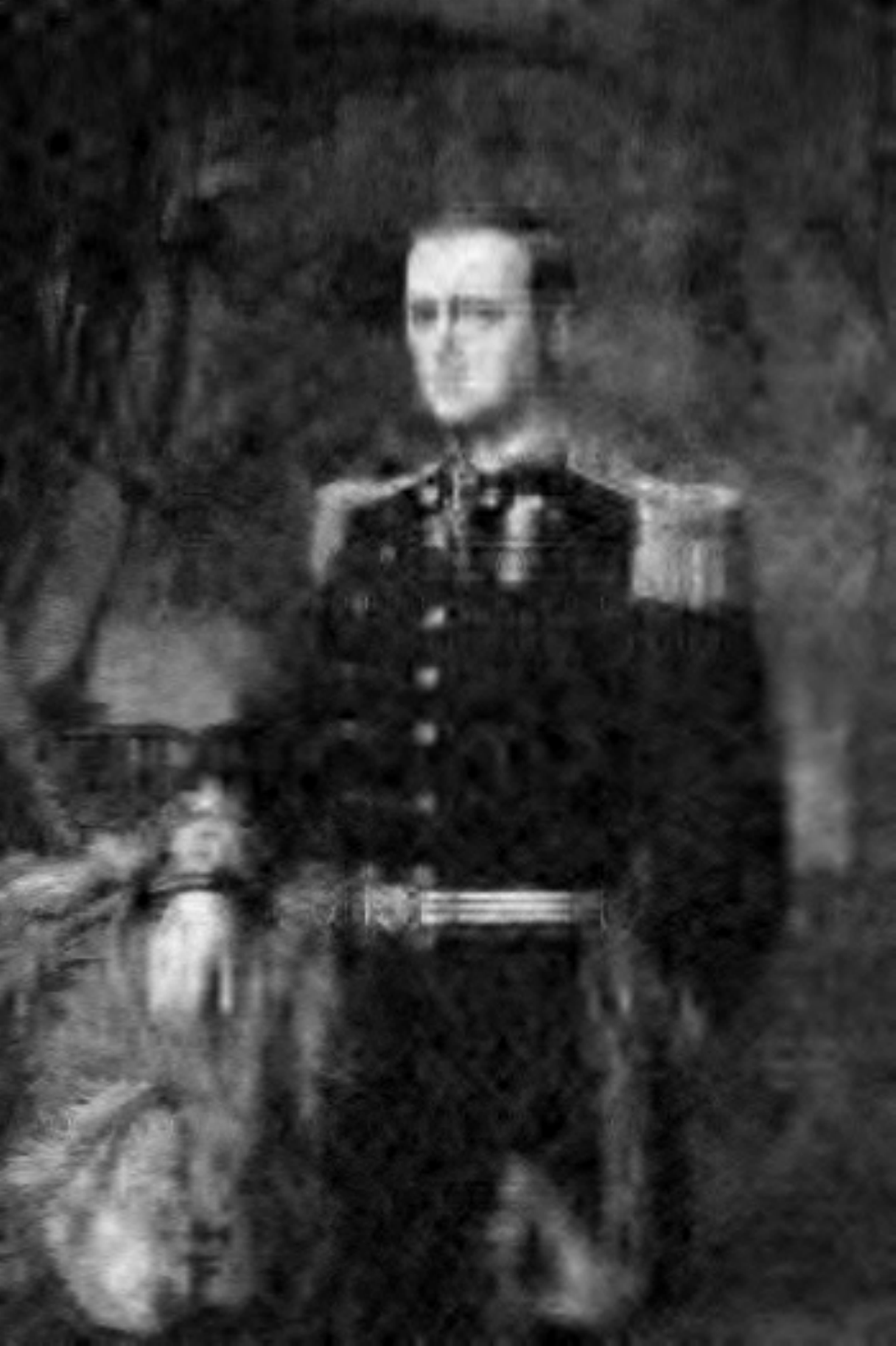}
\includegraphics[scale=0.2098]{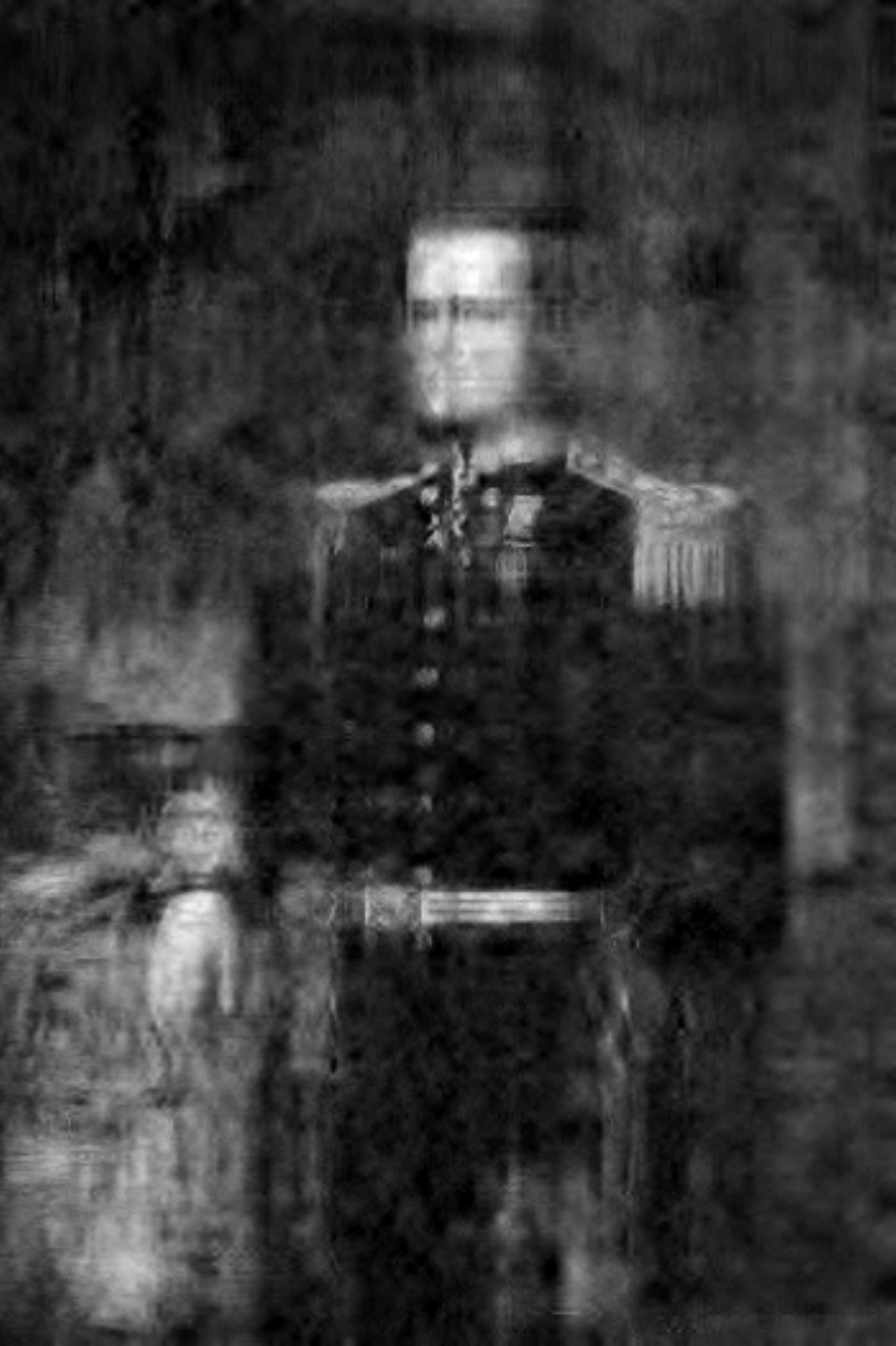}
\includegraphics[scale=0.2098]{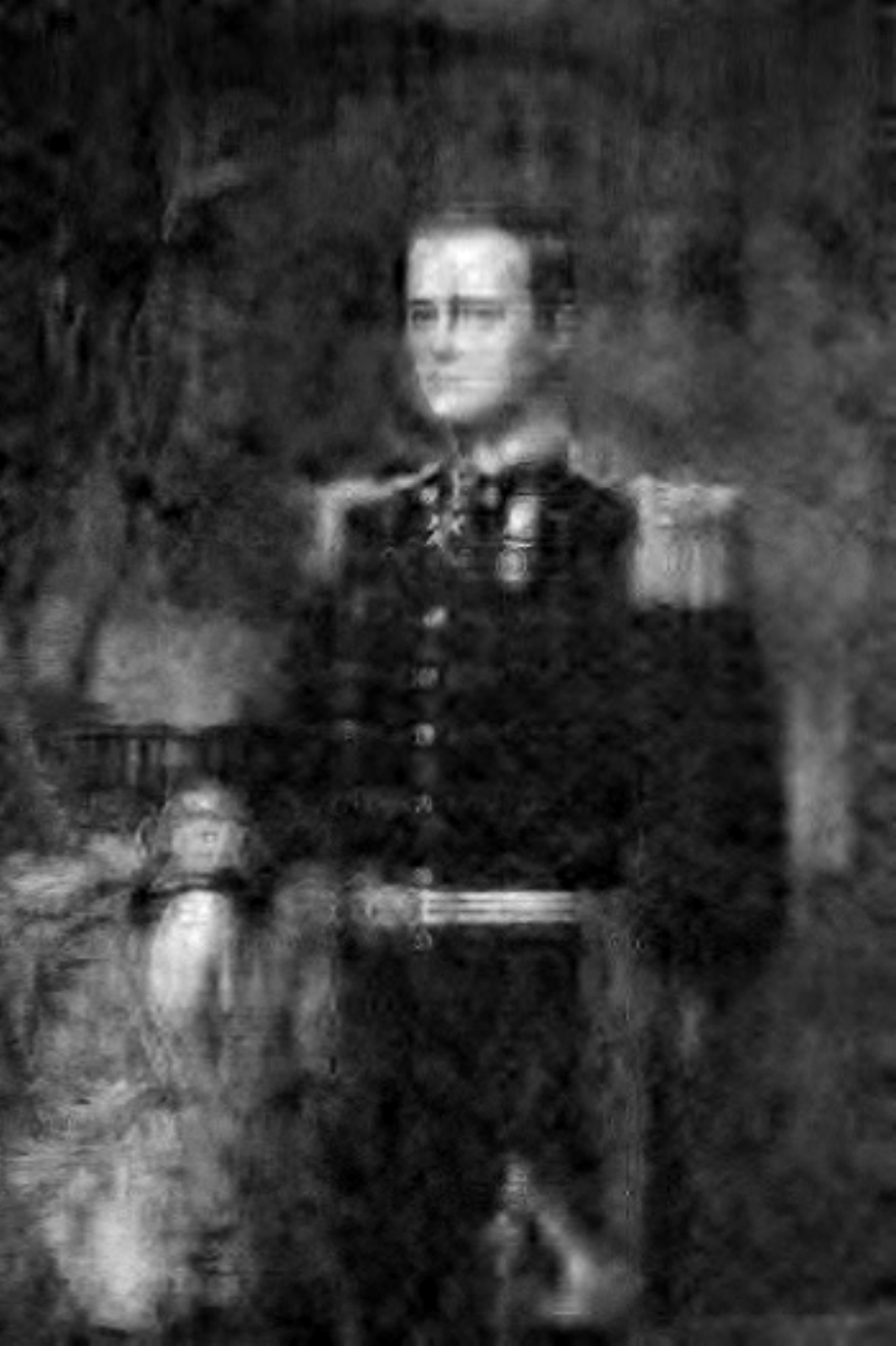}
\includegraphics[scale=0.33]{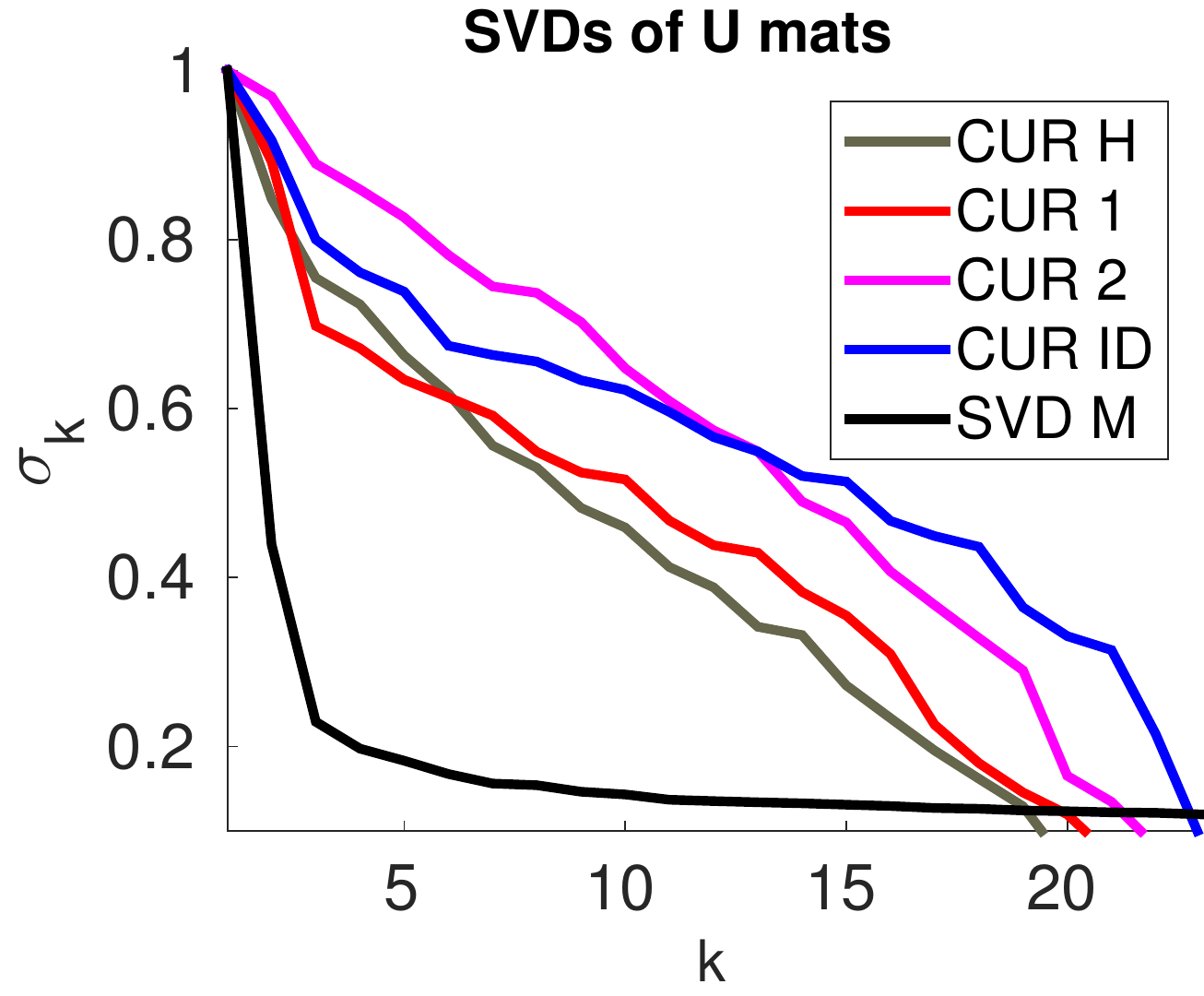}
}
\caption{Reconstructed images with \CUR~compression of the wavelet transformed 
image. Images resulting from applying Inverse Wavelet transform to matrix product 
$\mtx{CUR}$ obtained with
\CUR-1 in column 1, \CUR-H in column 2, \CUR-2 in column 3, and with \CURID~in column 4. 
Column 5 plots: singular value distributions of output $\mtx{U}$ matrices with the 
different algorithms compared.}
\label{fig:setV}
\end{figure*}

\section{Conclusions}

This paper presents efficient algorithms for computing \ID~and \CUR~decompositions.
The algorithms are obtained by very minor modifications to the classical pivoted 
\QR~factorization. As a result, the new \CURID~algorithm provides a direct and efficient 
way to  compute the \CUR~factorization using standard library functions, as provided in, 
e.g., BLAS and LAPACK. 

Numerical tests illustrate that the new algorithm \CURID~leads to substantially smaller 
approximation errors than methods that select the rows and columns based on leverage scores 
only. The accuracy of the new scheme is comparable to existing schemes that rely on additional
information in the leading singular vectors, such as, e.g., the DEIM-CUR \cite{2014arXiv1407.5516S}
of Sorensen and Embree, or the ``orthogonal top scores'' technique in the package rCUR.
However, we argue that \CURID~has a distinct advantage in that it can easily be coded up
using existing software packages, and our numerical experiments indicate an advantage in
terms of computational speed.

The paper also demonstrates that the two-sided \ID~is superior to the \CUR-decomposition
in terms of both approximation errors and conditioning of the factorization. The \ID~offers
the same benefits as the \CUR~decomposition in terms of data interpretation. However, for
very large and very sparse matrices, the \CUR~decomposition can be more memory efficient
than the \ID.

Finally, the paper demonstrates that randomization can be used to very substantially accelerate
algorithms for computing the \ID~and \CUR-decompositions, including techniques based on
leverage scores, the DEIM-CUR algorithm, and the newly proposed \CURID. Moreover, randomization
can be used to reduce the overall complexity of the \CURID-algorithm from $O(mnk)$ to $O(k^{2}m + k^{2}n + mn\log k)$.

%\vspace{2.mm}

\noindent
\textbf{Acknowledgement}
The research reported was supported by the Defense Advanced Projects Research
Agency under the contract N66001-13-1-4050, and by the National Science Foundation
under contracts 1320652 and 0748488.

\vspace{5.mm}
\bibliographystyle{plain}
\bibliography{ref}

\end{document}